\documentclass[reqno,11pt]{amsart}
\usepackage[colorlinks=true, linkcolor=blue, citecolor=blue]{hyperref}

\usepackage{amssymb}
\usepackage{amsmath, graphicx, rotating} 
\usepackage{color}
\usepackage{soul}
\usepackage[dvipsnames]{xcolor}
  
\usepackage{ifthen}
\usepackage{xkeyval}
\usepackage{todonotes}
\usepackage[T1]{fontenc}
\usepackage{lmodern}
\usepackage[english]{babel}

\usepackage{ upgreek }
\usepackage{stmaryrd}
\SetSymbolFont{stmry}{bold}{U}{stmry}{m}{n}
\usepackage{amsthm}
\usepackage{float}

\usepackage{ bbm }
\usepackage{ stmaryrd }
\usepackage{ mathrsfs }
\usepackage{ frcursive }
\usepackage{ comment }

\usepackage{pgf, tikz}
\usetikzlibrary{shapes}
\usepackage{varioref}
\usepackage{enumitem}
\usepackage{longtable}

\usepackage{mathtools}

\usepackage{dsfont}

\definecolor{rouge}{rgb}{0.7,0.00,0.00}
\definecolor{vert}{rgb}{0.00,0.5,0.00}
\definecolor{bleu}{rgb}{0.00,0.00,0.8}

\usepackage[margin=1.26in]{geometry}

\newtheorem{theorem}{Theorem}[section]
\newtheorem*{theorem*}{Theorem}
\newtheorem{lemma}[theorem]{Lemma}

\newtheorem{corollary}[theorem]{Corollary}
\newtheorem{proposition}[theorem]{Proposition}

\labelformat{hypothesis}{\textbf{M\kern-0.1mm#1}}

\newtheorem{condition}{Condition}
\newtheorem{conditionA}{A\kern-0.1mm}
\labelformat{conditionA}{\textbf{A\kern-0.1mm#1}}

\newtheorem{conditionEmpty}{\kern-0.1mm}
\labelformat{conditionEmpty}{\textbf{\kern-0.1mm#1}}

\renewcommand\dots{\hbox to 1em{.\hss.\hss.}}

\theoremstyle{definition}

\numberwithin{equation}{section}

\newcommand*{\abs}[1]{\left\lvert#1\right\rvert}
\newcommand*{\norm}[1]{\left\lVert#1\right\rVert}

\newcommand*{\scal}[2]{\left\langle {#1}, {#2} \right\rangle}

\def\bb#1{\mathbb{#1}}

\def\bf#1{\mathbf{#1}}
\def\scr#1{\mathscr{#1}}

\def\ds#1{\mathds{#1}}

\newcommand\ee{\varepsilon}

\newcommand{\Rdpint}{ (\bb R^d_+)^\circ}

\DeclareMathOperator{\supp}{supp}

\def\Rd {\mathbb{R}^d}
\def\Rd*{(\mathbb{R}^d)^*}
\def\Pd{{\mathbb{P}}^{d-1}}
\def\Pd*{(\mathbb{P}^{d-1})^*}

\begin{document}

\title[Tails of the invariant measure]  
{On the tails of the invariant measure for\\multidimensional affine stochastic recursions\\in the critical case}

\author{Ion Grama}
\author{Sebastian Mentemeier}\thanks{S.M. was supported by DFG grant ME 4473/2-1.} 
\author{Hui Xiao}\thanks{H.X. was supported by the National Natural Science Foundation of China 
(Grant Nos.\  12288201, 12595283 and 12631006)}

\curraddr[Grama, I.]{Univ Bretagne Sud, CNRS UMR 6205, LMBA, Vannes, France}
\email{ion.grama@univ-ubs.fr}

\curraddr[Mentemeier, S.]{Universität Hildesheim, Institut für Mathematik und Angewandte Informatik, Hildesheim, Germany}
\email{mentemeier@uni-hildesheim.de}

\curraddr[Xiao, H.]{State Key Laboratory of Mathematical Sciences, Academy of Mathematics and Systems Science, Chinese Academy of Sciences, Beijing 100190, China.}
\email{xiaohui@amss.ac.cn}

%%%%%%%%%%%%%%%%%%%%%%%%%%%%%%%%%%%%%%%%%%%%%%%%%%%%%%%%

\begin{abstract}
We study the  behavior at infinity of the invariant Radon
measure for the multidimensional affine stochastic recursion 
$V_n = A_n V_{n-1} + B_n,$ 
where $(A_n)_{n \geq 1}$ are positive random matrices, $(B_n)_{n \geq 1}$ are random vectors with nonnegative entries,
 and $(A_n, B_n)_{n \geq 1}$ are independent and identically distributed. 
%In the contractive regime where the top Lyapunov exponent of the random matrix products $A_n \cdots A_1$ is negative, 
%the invariant measure is a probability measure and has a power-law tail, a phenomenon described by the
%Kesten-Goldie theory.
 In the critical regime where the top Lyapunov exponent of the random matrix products $A_n \cdots A_1$ is zero, 
Brofferio, Peign\'e and Pham \cite{BPP21} recently established the existence and uniqueness,
up to multiplication by a constant, of an invariant Radon measure with infinite total mass.
They proved that the tail behavior of this measure when applied to radial sets is governed by a slowly varying function. 
Our goal is to show that this slowly varying function is actually bounded. Moreover, we investigate directional tail behavior. 
\end{abstract}

\date{\today}

\maketitle

%\tableofcontents

%%%%%%%%%%%%%%%%%%%%%%%%%%%%%%%%%%%%%%%%%%%%%%%%%%%%%%%%%
%%%%%%%%%%%%%%%%%%%%%%%%%%%%%%%%%%%%%%%%%%%%%%%%%%%%%%%%%
\section{Introduction}

The affine stochastic recursion is a basic model for stochastic systems in which multiplicative
and additive effects interact. 
Fix an integer $d \geq 1$ and 
let $(A_n, B_n)_{n \geq 1}$ be a sequence of independent and identically distributed copies of a random pair $(A,B)$,
where $A$ is a random $d\times d$ matrix with entries in $\bb R_+ = [0, \infty)$ 
and $B$ is a random vector in the positive orthant 
$\bb R_+^d = \{ (x_1, \ldots, x_d): x_i \in \bb R_+, \  \forall \  1 \leq i \leq d \}$ satisfying $\bb P(B \neq 0)>0$.  
No independence between $A$ and $B$ is assumed. 
In this paper, we consider the multidimensional affine stochastic recursion defined by 
\begin{align*}%\label{}
V_n = A_n V_{n-1} + B_n. 
\end{align*}  
This model arises naturally in a variety of probabilistic settings,
including perpetuities, stochastic volatility models, random difference equations,
and iterated random function systems, see the textbook \cite{BDM16} for a survey. 
A fundamental problem is to understand the long-time behavior 
of the Markov chain $(V_n)_{n \ge 0}$ and, in particular, the structure and asymptotic
properties of its invariant measure $\nu$ on $\bb R_+^d$.  
By definition, such a measure satisfies,  
for every compactly supported measurable function $f$ on $\bb R_+^d$, 
\begin{align}\label{def invariant measure}
\nu(f) = \bb E \int_{\bb R_+^d} f(Ax+B) \, \nu(dx),
\end{align}
where the expectation is taken with respect to the law $\mu$ of $(A, B)$.

The asymptotic behavior of the invariant measure is governed to a large extent by the
top Lyapunov exponent $\lambda$ of the associated products of random matrices (cf.\ \cite{FK60}): 
under a suitable integrability assumption, $\bb P$-almost surely, 
\begin{align}\label{def Lyapunov exponent}
\lambda = \lim_{n \to \infty} \frac{1}{n} \log \norm{A_n \cdots A_1}, 
\end{align}
where $\| A \| = \sup_{x \in \bb R^d_+ \setminus \{0\} } \frac{|Ax|}{|x|}$, 
% is the operator norm of a $d \times d$ matrix $g$,
$|x| = \sum_{i=1}^d |\langle x, e_i\rangle |$ and $(e_i)_{1 \leq i \leq d}$ is the standard basis of $\bb R^d$. 
There are two typical regimes in the literature. 
In the {\it contractive regime} $\lambda <0$,  the Markov chain $(V_n)_{n \ge 0}$ 
%When $\gamma<0$ (called {\it contractive case}), the system is contractive and 
admits a unique invariant probability measure; see for example \cite[Theorem 4.1.4]{BDM16}. 
A remarkable phenomenon in this regime is that the invariant distribution may exhibit heavy tails even when the noise variables
themselves have light tails. 
More specifically, Kesten's theorem (\cite{Kes73}) tells us that, under some extra assumptions, 
there are constants $\alpha, C>0$ and a continuous strictly positive function $r$ on $\bb S^{d-1}_+$ such that, 
for every $x \in \bb S^{d-1}_+$ and $b > a > 0$, 
\begin{align}\label{eq:limit kesten}
\lim_{t \to \infty} t^\alpha \nu \left( \left\{ z \in  \bb R^{d}_+ \, : \, \scal{x}{z} \in t[a,b]  \right\}  \right) 
= C \left( a^{-\alpha} - b^{-\alpha} \right) r(x) =  C r(x) \int_a^b \frac{\alpha}{s^{1+\alpha}} ds. 
\end{align}
Here and below, we write $t [a, b] = [t a, t b]$ and denote by $\bb S^{d-1}_+=\{ x \in \bb R^d_+ \, : \, |x|=1\}$ the intersection of the unit sphere with the positive orthant. 
Kesten's work has stimulated extensive subsequent work and we refer to the textbook \cite{BDM16} for a survey of the field.

In the {\it critical regime} $\lambda = 0$, the contraction mechanism disappears and the situation is fundamentally different. 
It turns out that the invariant object is no longer a probability measure but an infinite Radon measure.
In the one-dimensional case $d=1$, the existence and uniqueness, up to scalar multiplication, of $\nu$ was proved by Babillot, Bougerol and Elie \cite{BBE97}, see also \cite{AI23,BB15} for recent contributions. In \cite{BBE97}, the asymptotic behavior of this measure was also studied, with precise asymptotics being obtained by  Buraczewski \cite{Bur07} using renewal methods and Poisson equations. 
More specifically, the invariant measure satisfies a logarithmic tail homogeneity property: 
there is a constant $C>0$ such that for any $b > a > 0$, 
\begin{align}\label{eq:limit critical 1d}
\lim_{t \to \infty} \nu \left(  \left\{ z  \in \bb R \, : \, |z| \in t [a,b]  \right\}  \right) 
= C \log \frac{b}{a} = C\int_a^b \frac{1}{s} \, ds. 
\end{align}
The natural limiting measure in the critical regime is therefore the Haar measure
$ds/s$ on the multiplicative group $(0,\infty)$. See also \cite{BBD12, BB15, Kol13} for extensions to more general recursions.

For higher dimensions $d \geq 2$, the study is considerably more delicate and requires new ideas.  
The main difficulties arise from the non-commutativity of matrices and the need for a
Markov renewal theory for centered Markov random walks driven by products of random matrices. 
It was proved recently in \cite[Theorem 1.1]{BPP21} 
that  the Markov chain $(V_n)_{n \ge 0}$ possesses an -- up to scaling -- unique invariant Radon measure $\nu$ 
on $\bb R_+^d$ with infinite total mass, see also \cite{ABP24} 
for a study of recurrence properties of $(V_n)_{n \ge 0}$ under relaxed assumptions.  
Subject to conditions to be detailed below, 
it was also proved in \cite[Theorem 1.1]{BPP21} that 
there exist a positive slowly varying function $L: \bb R_+ \to \bb R_+$ and constants $a,b,c>0$ such that for any $t \geq 1$, 
\begin{equation}\label{eq:result.BPP}
L(t) \le \nu \left(  \left\{ x \in \bb R^d_+ \, : \,  |x| \in t [a, b]  \right\} \right) \le c L(t).
\end{equation}
A function $L: \bb R_+ \to \bb R_+$ is called slowly varying at infinity if $\lim_{t \to \infty} \frac{L(ts)}{L(t)} = 1$ for any $s>0$. 
We note that the interval $[a,b]$ is fixed in the statement \eqref{eq:result.BPP}. 

Using recent progress on the renewal theory and duality theory for centered products of random matrices in \cite{GMX25}, 
our first main result, stated as Theorem \ref{thm:main3} below, 
proves that the slowly varying function $L(t)$ is indeed bounded. 
Consequently, for each fixed interval $[a, b]$, the radial mass $\nu ( \{ z  \in \bb R^d_+ \, : \, |z| \in t [a,b] \} )$
 remains bounded from above and below by positive constants as $t \to \infty$.
Moreover, our main theorem (see Theorem \ref{thm:main2} below) establishes a directional analogue of this result. 
 To be precise, consider the Markov chain on $\bb S^{d-1}_+$  defined by the recursion 
$X_n: = \frac{A_n^{\top} X_{n-1}}{|A_n^{\top} X_{n-1}|}$. 
Under the assumptions of the paper, it possesses a unique invariant probability measure $\pi$.  
We are going to prove that for every $x$ in the support of $\pi$, there are constants $0<c_1\le C_1<\infty$ such that for all $0<a<b$, 
\begin{align}\label{Intro-entry-asy}
c_1 \log \frac{b}{a} & \le \liminf_{t \to \infty} \nu \left(  \left\{ z  \in \bb R^d_+ \, : \, \scal{x}{z} \in t [a,b]  \right\}  \right)  \notag\\
& \le \limsup_{t \to \infty} \nu \left(  \left\{ z  \in \bb R^d_+ \, : \,  \scal{x}{z} \in t [a,b] \right\}  \right) \le C_1 \log \frac{b}{a}. 
\end{align}

In the one-dimensional case, and in the case where $A$ belongs to the group of similarity matrices 
(product of a dilation coefficient and of an orthogonal matrix), 
it was proved in \cite{BBD12, Bur07} and \cite{Kol13}, respectively, that \eqref{Intro-entry-asy} holds with $c_1 =C_1$, 
which corresponds to reaching the exact limit as in \eqref{eq:limit critical 1d} for $\nu ( \{ z \in \bb R^d_+ \, : \,  \scal{x}{z} \in t [a, b]  \} )$. 
Obtaining a precise limit in the present multivariate setting would require some exact asymptotics 
in the conditioned central and local limit theorems for Markov random walks, which are currently out of reach. 

The remainder of the paper is organized as follows. In Section \ref{Sec main results} we introduce the necessary 
notation and assumptions and state the main results. Section \ref{sec-prelimin} collects the
probabilistic and renewal-theoretic ingredients needed in the sequel, including the
estimates for the associated centered Markov random walk. 
The proofs of the main results are given in Section \ref{Section-proof}.

%%%%%%%%%%%%%%%%%%%%%%%%%%%%%%%%%%%%%%%%%%%%%%%
%%%%%%%%%%%%%%%%%%%%%%%%%%%%%%%%%%%%%%%%%%%%%%%

\section{Notation, assumptions and main results}\label{Sec main results}

%%%%%%%%%%%%%%%%%%%%%%%%%%%%%%%%%%%%%%%%%%%%%%%
%%%%%%%%%%%%%%%%%%%%%%%%%%%%%%%%%%%%%%%%%%%%%%%

\subsection{Notation} 
A $d \times d$ matrix $g$ with nonnegative entries is called \emph{allowable} 
if every row and every column of $g$ has at least one strictly positive entry. 
Let $\bb M$ be the semigroup of allowable nonnegative $d\times d$ matrices, 
and let $\bb M_+$ be the subsemigroup of $\bb M$ with strictly positive entries. 
By the Perron-Frobenius theorem, 
every $g \in \bb M_+$ admits an algebraically simple positive eigenvalue $\lambda_g$ equal to its
spectral radius, with a corresponding unit eigenvector $v_g$ having strictly positive coordinates. 
The action of $g \in \bb M$ on a vector $v \in \bb R_+^d$ is denoted by $gv$.
Recall that, for any $v \in \bb R^d$, the associated norm is denoted by $|v| = \sum_{i=1}^d |\langle v, e_i\rangle |$,
where $(e_i)_{1 \leq i \leq d}$ is an orthonormal basis of $\bb R^d$. 
For any $g \in \bb M$, set $\| g \| = \sup_{v \in \bb R^d_+ \setminus \{0\} } \frac{|gv|}{|v|}$
and $\iota(g)  = \inf_{v \in \bb R^d_+ \setminus \{0\} } \frac{|gv|}{|v|}$. 
In terms of the entries of $g$, we have
\begin{align}\label{def-norm-g-iota}
\| g \| = \max_{1 \leq j \leq d} \sum_{i = 1}^d  g_{i, j}
\quad 
\mbox{and}
\quad
\iota(g) = \min_{1 \leq j \leq d} \sum_{i = 1}^d  g_{i, j}. 
\end{align}
We also write $N(g)= \max\{ \iota(g)^{-1}, \|g\|\}$.

The topological interior and closure of a set $E \subset \bb R^d$ are denoted by $E^\circ$ and $\overline{E}$, respectively. 
The unit sphere in $\bb R^d$ is denoted by $\bb S^{d-1}=\{ v \in \bb R^d : | v | =1 \},$
and we set  
%Let  $\bb S_+^{d-1}$ be the intersection of the unit sphere with the positive quadrant: 
$\bb S_+^{d-1} = \bb S^{d-1} \cap \bb R_+^{d}$.  
%The interior of $\bb S_+^{d-1}$ is denoted by $(\bb S_+^{d-1})^\circ$. 
%The interior of a measurable set $D$ is denoted by $D^\circ$. 
For any  $g \in \bb M$, we can define its  action on $x \in \bb S_+^{d-1} $ by $g \cdot x = \frac{g x}{|gx|}$. 
The allowability of $g$ ensures that $gx \neq 0$, so this action is well defined.

Denote by $\textrm{Aff}(\bb R_+^d) = \bb M \ltimes \bb R_+^d$ the semigroup of affine mappings of $\bb R_+^d$. 
Let $(A_n, B_n)_{n\geq 1}$ be a sequence of independent and identically distributed $\textrm{Aff}(\bb R_+^d)$-valued random elements on 
a probability space $(\Omega,\mathscr F,\bb P)$, 
with common law $\mu$ and the generic pair $(A,B)$.
Denote by $\overline \mu$ the
marginal law of $A$ %with values in $\bb M$, 
and write $\supp \overline\mu$ for its support. % of the measure $\overline\mu$

The  Markov chain on $\bb S^{d-1}_+$ relevant to the tail problem is naturally
associated with the transposed matrices. 
For $x \in \bb S^{d-1}_+$, let $X_0^x=x$, $S_0^x=0$ and, for $n\geq 1$,  
\begin{equation}\label{def-Xn-Sn}
	 X_n^x = (A_n^\top \cdots A_1^\top) \cdot x, \qquad S_n^x = \log |A_n ^\top \cdots A_1^\top x|, 
\end{equation}
where $A_n^\top$ is the transpose of the matrix $A_n$. 
%Under the probability $\bb P$, for any $x\in \bb S^{d-1}_+$, the sequence 
Thus $(X_n^x)_{n\geq 0}$ is a Markov chain on $\bb S^{d-1}_+$,
while $(X_n^x, S_n^x)_{n\geq 0}$ 
%new Markov chain on $\bb S^{d-1}_+ \times \bb R_+$ called Markov random walk. 
is a Markov random walk on $\mathbb S_+^{d-1}\times\mathbb R$.
The second coordinate is referred to as the additive component of the Markov random walk. 
%The second coordinate will be referred to as the additive component
%of the Markov random walk. 
%Denote by $\bb P_x$ and $\bb E_x$ the probability measure and the expectation 
%induced by the Markov walk $(X_n^x)$ on the canonical space $(\bb S^{d-1}_+)^{\bb N}$. 
%Therefore, under the measure  $\bb P_x$ and the expectation $\bb E_x$,  
%we can omit the superscript $x$ in the writing of $X_n^x$
%and in that of $S_n^x$. 

\subsection{Assumptions}

Let $\Gamma$ be the closed subsemigroup of nonnegative $d \times d$ matrices generated by $\supp \overline\mu$. 
We require the following allowability and positivity condition. 
%The set $\Gamma$ is said to be allowable if every $g\in \Gamma$ is allowable.

\begin{conditionA}\label{proximality}
Every $g \in \Gamma$ is allowable and $\Gamma$ contains at least one strictly positive matrix.
\end{conditionA}

We also impose a moment and nondegeneracy assumption on the pair $(A,B)$:
\begin{conditionA}\label{Condition-moments-A-B} 
There exists a constant $\delta>0$ such that 
$\bb E \big[ N(A)^{\delta}\big]<\infty$ and $\bb P\big( \iota(A) \ge 1 + \delta \big)>0$ as well as $\bb E (|B|^{\delta}) < \infty$ and $\bb E (|B|^{-\delta}) < \infty$. 
\end{conditionA}

Note that assumptions \ref{proximality} and \ref{Condition-moments-A-B} guarantee the existence and finiteness of 
the Lyapunov exponent $\lambda$ defined by \eqref{def Lyapunov exponent}, see for example \cite[Theorem 6.1]{BDGM14}.

The assumption $\bb P\big( \iota(A) \ge 1 + \delta \big)>0$ in \ref{Condition-moments-A-B} ensures that 
the elements in the support of $(A, B)$ have no common fixed point on $\bb R_+^d$, i.e., 
for any $x \in \bb R_+^d$, we have $\bb P (Ax + B = x) < 1$. 
Indeed, take such $A$ with $\iota(A) \geq 1 + \delta$, which means that for any $x \in \bb R_+^d$, 
it holds that $|Ax| \geq (1 + \delta) |x|$. 
For $x \neq 0$, since $B \in \bb R_+^d$, we have $\lvert Ax+B \rvert \ge \lvert Ax \rvert > \lvert x \rvert,$ 
so $Ax+B=x$ is impossible on this event. 
For $x=0$, we use that $\bb P(B \neq 0)=1$ due to \ref{Condition-moments-A-B}.

In addition, we impose the following non-arithmeticity assumption. 

\begin{conditionA}\label{Condition-dense}
The additive subgroup of $\bb R$ generated by $\{ \log \lambda_g \, : \,  g \in \Gamma \cap \bb M_+\}$ is dense in $\bb R$.
\end{conditionA}

As stated in the introduction, the present paper is devoted to the
critical regime.
\begin{conditionA}\label{Condition-Lyapunov-zero}
The Lyapunov exponent vanishes: $\lambda = 0$. 
%$$ \gamma =\lim_{n \to \infty} \frac{1}{n} \log \norm{A_n \cdots A_1} =0.$$
\end{conditionA}

We further require the following {\em Furstenberg-Kesten} condition on the entries of the matrices.
\begin{conditionA}\label{Condition-Furstenberg-Kesten}
There exists a constant $\varkappa > 1$ such that for all $g = (g_{i,j})_{1 \le i,j \le d} \in \supp \overline\mu$, 
\begin{equation*}
\frac{\max_{1 \le i,j \le d} g_{i,j}}{\min_{1 \le i,j \le d} g_{i,j}} \le  \varkappa.
\end{equation*}
\end{conditionA}
This condition enables us to employ the renewal theoretic results from \cite[Section 7]{GMX25}. 
Note also that assumption \ref{Condition-Furstenberg-Kesten} implies that 
$\Gamma \subset S_{\varkappa^{-2}}$ in the notation of \cite{BPP21}, 
see the comments after \cite[Lemma 2.1]{BPP21}.
%\todo{Check that we can use the constant $\varkappa$ in \ref{Condition-Furstenberg-Kesten}}
It is worth mentioning that under condition \ref{Condition-Furstenberg-Kesten}, 
we have $\| g \| \leq \varkappa \, \iota(g)$, so that $\iota(g)^{-1} \leq \varkappa \|g\|^{-1}$. 
%{\color{magenta}Hence, in order to have $\bb{E} [ \exp (\delta \iota(A)^{-1})]<\infty$ in \ref{Condition-moments-A-B}, 
%it suffices to assume $\bb{E} [ \exp (\delta \varkappa \|g\|^{-1})] < \infty$. }

Finally, we need the following nondegeneracy assumption. 

\begin{conditionA}\label{Condition-Irreducibility}
There is no affine subspace $W \subset \bb R^d$ such that $W \cap \bb R_+^d$ is non-empty, bounded and invariant under the action of all $g \in \supp \bar{\mu}$.
\end{conditionA}

Under assumptions \ref{proximality}-\ref{Condition-Irreducibility}, 
\cite[Theorem 1.1]{BPP21} provides us with the existence and uniqueness, up to multiplication by a positive constant,  
of an invariant Radon measure $\nu$ on $\bb R_+^d$ for the Markov chain $(V_n)_{n \geq 0}$,
that is, a measure satisfying \eqref{def invariant measure}. 
The measure $\nu$ has infinite total mass and satisfies \eqref{eq:result.BPP} for some constants $a, b, c>0$ and a positive slowly varying function $L$.

\subsection{Main results}

It follows for instance from \cite[Theorem 2.1]{HH08}
that, under \ref{proximality} and \ref{Condition-moments-A-B},
%if $\int \log N(g) \mu(dg) < \infty$, $\Gamma$ is allowable and contains a strictly positive matrix, 
 there is a unique stationary probability measure $\pi$ for the Markov chain $(X_n^x)_{n \geq 0}$ on $\bb S^{d-1}_+$. 
Its support is given by  
 $$ \overline{\left\{ v_{g^\top} \in \bb S^{d-1}_+ \, : \, g \in  \Gamma \cap \bb M_+  \right\}},$$ 
see \cite[Lemma 4.3]{BDGM14} for a proof.

Our first main result gives a uniform logarithmic description of the
tail of $\nu$ along every direction in $\operatorname{supp}\pi$. 

\begin{theorem}\label{thm:main2}
Assume \ref{proximality}-\ref{Condition-Irreducibility}. 
Then, there exist constants $0<c_1\le C_1<\infty$ such that for all $0<a<b$ 
%\begin{align}\label{eq:our.result.norm}
%c_1 \log \frac{b}{a} & \le \liminf_{t \to \infty} \nu \left(  \left\{ z  \in \bb R^d_+ \, : \, |z| \in t [a,b]  \right\}  \right)  \notag\\
%& \le \limsup_{t \to \infty} \nu \left(  \left\{ z  \in \bb R^d_+ \, : \,  |z| \in t [a,b] \right\}  \right) \le C_1 \log \frac{b}{a}, 
%\end{align}
and $x \in \supp \pi$, 
\begin{align}\label{eq:our.result.scalar}
c_1 \log \frac{b}{a} & \le \liminf_{t \to \infty} \nu \left(  \left\{ z  \in \bb R^d_+ \, : \, \scal{x}{z} \in t [a, b]  \right\}  \right)  \notag\\
& \le \limsup_{t \to \infty} \nu \left(  \left\{ z  \in \bb R^d_+ \, : \,  \scal{x}{z} \in t [a, b] \right\}  \right) \le C_1 \log \frac{b}{a}. 
\end{align}
\end{theorem}

Note that our result \eqref{eq:our.result.scalar} can be understood as the analogue of \eqref{eq:limit kesten} for $\alpha=0$, recalling that $\int_a^b \tfrac{ds}{s}=\log \tfrac{b}{a}$. Indeed, the critical case corresponds to $\alpha=0$, in the following sense. 
The function $m(s) = \lim_{n \to \infty} \bb{E} (\norm{A_n \cdots A_1}^s)^{1/n}$ is strictly convex with $m(0)=1$ and $m'(0)=\lambda$ (the top Lyapunov exponent in \eqref{def Lyapunov exponent}), see \cite[Theorem 6.1]{BDGM14}. In the contractive case, $m'(0)<0$ and $\alpha$ is defined as the unique positive value with $m(\alpha)=1$, given its existence. In the critical case, $m'(0)=0$, and $0$ is the only solution of the equation $m(s)=1$. In this sense, one can assign the value $\alpha=0$ to the critical case.

There is a simple situation in which Theorem \ref{thm:main2}
immediately yields corresponding estimates for radial annuli.  
If the vector $(1/d, \ldots, 1/d) \in \bb S^{d-1}_+$ is in $\supp \pi$, 
then, after a rescaling of the interval, 
Theorem \ref{thm:main2} applies directly to $\nu ( \{ z  \in \bb R^d_+ \, : \, |z| \in t [a, b]  \} )$, 
thereby extending \eqref{eq:result.BPP}. 
If this is not the case, we still have the following result.

\begin{theorem}\label{thm:main3}
Assume \ref{proximality}-\ref{Condition-Irreducibility}. 
Then, for any $0<a<b$, there are constants $c_1\le C_1<\infty$ such that 
\begin{align*}%\label{eq:our.result.norm02}
c_1 & \le \liminf_{t \to \infty} \nu \left(  \left\{ z  \in \bb R^d_+ \, : \, |z| \in t [a,b]  \right\}  \right)  \notag\\
& \le \limsup_{t \to \infty} \nu \left(  \left\{ z  \in \bb R^d_+ \, : \,  |z| \in t [a,b] \right\}  \right) \le C_1.  
\end{align*}
Moreover, a sufficient condition for $c_1$ to be positive is $b/a > \varkappa d$. 
\end{theorem}

In particular, the mass assigned by $\nu$ to every fixed
multiplicative annulus remains bounded from above as the annulus is pushed to infinity. 
If, in addition, $b/a > \varkappa d$, then it is also bounded away from zero.
In this sense, Theorem \ref{thm:main3} removes the possible unbounded
slowly varying fluctuations left open by \eqref{eq:result.BPP}.

\section{Preliminaries}\label{sec-prelimin}

From now on, we assume \ref{proximality}-\ref{Condition-Irreducibility}
and use the notation introduced in Section \ref{Sec main results}. We show an integrability estimate for the
invariant Radon measure, introduce the Poisson equation associated with the tail problem, 
and state the renewal estimates for the 
centered Markov random walk that will be used in the proof of the main results.

\subsection{Integrability estimate for the invariant measure}
We begin with a consequence of the upper bound in \eqref{eq:result.BPP} and our moment assumption \ref{Condition-moments-A-B}. 
Although the invariant measure $\nu$ has infinite total mass, it has finite negative  moments both at infinity and at the origin.

\begin{lemma}\label{Lem-bound.nu-z.gamma}
For any $\gamma \in (0, \delta/2)$, with $\delta$ from the moment assumption \ref{Condition-moments-A-B}, we have 
$$ \int_{\bb R^d_+}  \frac{1}{|z|^\gamma} \nu(dz) < \infty.$$
\end{lemma}

\begin{proof}
Let $0< a < b$ be the constants appearing in \eqref{eq:result.BPP} and set $r = \frac{b}{a} > 1$.  Denote $\bb B_+(0, c) = \{ z \in \bb R^d_+:  |z| \leq c \}$ for $c >0$.  
%Since $\nu$ is a Radon measure, there is a constant $c_1 > 0$ such that $\nu (\{z \in \bb R_+^d: |z| < a \}) \leq c_1$. 
We decompose the integral at $a$ and have
\begin{align*}%\label{}
 \int_{\bb R^d_+}  \frac{1}{|z|^\gamma} \nu(dz)  
& =  \sum_{n = 0}^{\infty}  \int_{\{z \in \bb R_+^d: |z| \in r^n [a, b)  \} }  \frac{1}{|z|^\gamma} \nu(dz)
+  \int_{\bb B_+(0, a) }  \frac{1}{|z|^\gamma} \nu(dz)   \notag\\
& = I_1 + I_2. 
\end{align*}
Concerning the first term $I_1$, by \eqref{eq:result.BPP}, there is $c_1 > 0$ 
such that $\nu (\{z \in \bb R_+^d: |z| \in r^n [a, b) \}) \leq c_1 L(r^n).$ 
Moreover, by Potter's theorem, the slowly varying function satisfies $L(t)\leq c_2 t^{\gamma/2}$ for some constant $c_2$ and all sufficiently large $t$. Consequently, using also that $\nu$ is a Radon measure, we have that 
\begin{align}\label{inequality nu annulus}
\nu \left( \left\{ z \in \bb R_+^d: |z| \in r^n [a, b)  \right\} \right) \leq c_3 (r^n)^{\gamma/2}. 
\end{align}
Hence,
\begin{align*}
I_1 \leq   \sum_{n = 0}^{\infty}  \frac{1}{(r^n a)^{\gamma}} \nu \left( \left\{ z \in \bb R_+^d: |z| \in r^n [a, b) \right\} \right) 
\leq   c_3  \sum_{n = 0}^{\infty}  \frac{1}{(r^n a)^{\gamma/2}}  < \infty. 
\end{align*}
%Note that if we work under the assumption that $|B|\geq c>0$, then $\nu(B_+(0,c))=0$ and it is sufficient to use the above bound for $I_1$.

Considering the second term $I_2$, observe that, for any $x \in \bb{R}^d_+$ and $0<t<a$, by writing $\bf 1 = (1, \ldots, 1)$, we have 
\begin{align*}
	|Ax+B|=\langle \bf 1,Ax+B\rangle \leq t 
	 & \quad  \Rightarrow \quad |B| \leq t \  \text{and} \ |Ax|=\langle \bf 1, Ax \rangle \leq t - \langle \bf 1, B \rangle = t-  |B|  \notag\\
	 & \quad  \Rightarrow \quad |B| \leq t \  \text{and} \ \iota(A)|x| \leq t-|B|  \notag\\
	 & \quad \Rightarrow  \quad |B| \leq t \  \text{and}  \  |x| \leq \frac{ t - |B| }{\iota(A)} \leq \frac{a}{\iota(A)}
\end{align*}
and, as a consequence of \eqref{def invariant measure}, for any $0<t<a$,
\begin{align}
	\nu(\bb B_+(0, t)) &= \bb E \int_{\bb R_+^d} \ds{1}_{[0,t]}(|Ax+B|) \, \nu(dx) \notag  \\
	&\leq \bb E \int_{\bb R_+^d} \ds{1}_{[0,t]}(|B|) \ds{1}_{[0,\iota(A)^{-1}a]} (|x|) \nu(dx) \notag \\
	& = \bb E \left[ \ds{1}_{[0,t]}(|B|) \, \nu \left( \bb B_+ \big( 0, \iota(A)^{-1}a \big) \right) \right]. \label{eq:bound nu at zero}
\end{align}
Note that this implies in particular that $\nu(\{0\})=0$. To further bound the behavior of $\nu$ close to the origin, we proceed as follows. 
On the event $\{\iota(A)^{-1} \leq 1 \}$, we simply use $\nu \left( \bb B_+ \big( 0, \iota(A)^{-1}a \big) \right) \leq \nu(\bb B_+(0,a)) \leq c$
for some constant $c>0$.
On the event $\{ \iota(A)^{-1} > 1 \}$, we use the bound \eqref{inequality nu annulus} to get that, 
with $r = \frac{b}{a}$ and $m = [- (\log \iota(A)) / \log \frac{b}{a}] + 1$,
\begin{align*}
\nu \left( \bb B_+ \big( 0,\iota(A)^{-1}a \big) \right)  
& \leq  \nu(\bb B_+(0,a)) +   \sum_{n=1}^{m} \nu \left( \left\{ z \in \bb R_+^d: |z| \in r^{n-1} [a, b)  \right\} \right)   \notag\\
& \leq  \nu(\bb B_+(0,a) ) +  c  \sum_{n=1}^{m}  r^{(n-1) \gamma/2}  \notag\\
& =   \nu(\bb B_+(0,a) ) +    c \frac{r^{m \gamma/2} -1}{r^{\gamma/2} - 1}  \notag\\
& \leq  c'  \left( 1 + \iota(A)^{- \gamma/2} \right). 
\end{align*}
Plugging this into \eqref{eq:bound nu at zero}, we have, using the Cauchy-Schwarz inequality,
\begin{align*}
\nu(\bb B_+(0,t) )
& \leq  c' \bb E \left[ \ds{1}_{[0,t]}(|B|) \,   \left( 1 + \iota(A)^{- \gamma/2} \right) \right]  \notag\\
&  \leq  c'  \bb P(|B| \leq t)^{1/2}  \left\{ \bb{E} \left[ \left( 1 + \iota(A)^{- \gamma/2} \right)^2 \right]  \right\}^{1/2}  \notag\\
& \leq 2 c' \bb P(|B| \leq t)^{1/2}  \big[ 1 + \bb E \big( \iota(A)^{- \gamma} \big) \big].
\end{align*}
By Assumption \ref{Condition-moments-A-B}, the second factor $\bb E (\iota(A)^{- \gamma})$ is finite if we choose $\gamma \leq \delta$. By the same assumption, $\bb{E}(|B|^{-\delta}) <\infty$ and hence, by an application of Markov's inequality,
we obtain that there exists a constant $c'' >0$ such that for any $0<t<a$, 
\begin{align}\label{nu on small ball radius t}
\nu(\bb B_+(0,t) ) \leq c'' \bb P(|B| \leq t)^{1/2} 
= c'' \bb P(|B|^{-1} \geq t^{-1})^{1/2} \leq c'' \left(  t^{\delta} \, \bb{E}(|B|^{-\delta}) \right)^{1/2}.
\end{align}
Now, using \eqref{nu on small ball radius t} we conclude that
\begin{align*}
I_2  & = \int_{\bb B_+(0, a) }  \frac{1}{|z|^\gamma} \nu(dz) \notag\\
& = \sum_{n=0}^\infty \int_{\{z \in \bb R_+^d: |z| \in [2^{-(n+1)}a,  \, 2^{-n} a)  \} } \frac{1}{|z|^\gamma} \nu(dz) \\
& \leq  \sum_{n=0}^\infty\frac{1}{(2^{-(n+1)}a)^\gamma }  \nu \left( \left\{z \in \bb R_+^d: |z| \in [2^{-(n+1)}a, 2^{-n} a)  \right\}  \right) \\
& \leq  \sum_{n=0}^\infty\frac{1}{(2^{-(n+1)}a)^\gamma }  \nu \big( \bb B_+(0,2^{-n}a) \big) \\
&  \leq c \sum_{n=0}^\infty\frac{1}{(2^{-(n+1)}a)^\gamma }  \frac{1}{(2^n/a)^{\delta/2}} <\infty,
\end{align*}
where the series is finite since $\gamma \in (0, \delta/2)$. 
This completes the proof of the lemma. 
\end{proof}

\subsection{The Poisson equation}

If $\bar \nu_x$ denotes the push-forward measure of $\nu$ under the map $z \mapsto \langle x,z \rangle$, 
then \eqref{eq:our.result.scalar} is a statement about dilations of $\bar \nu_x$. In particular, 
$$\nu \left( \left\{z \in \bb{R}^d_+:  \langle x, z \rangle \in e^s[a,b]  \right\} \right) 
= \bar \nu_x (e^s[a,b]) = \int_{\bb{R}_+}  \ds{1}_{[a,b]}(e^{-s}z) \, \bar\nu_x(dz).$$
We are going to study vague convergence of $\bar{\nu}_x(e^s \cdot)$, by evaluating its integral over a family of test functions.
We denote by $C^1_c(\bb R_+^*)$ the space of continuously
differentiable functions with compact support contained in $\bb R_+^* = (0, \infty)$ 
 and recall $\Rdpint =  (0, \infty)^d$. 
Let $\Phi \in C^1_c(\bb R_+^*)$ be nonnegative and
define
%For any nonnegative function $\Phi \in C^1_c(\bb R_+^*)$ where $\bb R_+^* = (0, \infty)$, 
%we consider 
\begin{align}\label{def-function-f-Phi}
f_\Phi(y) = \int_{\bb R^d_+}  \Phi\big( \langle y, z \rangle \big) \nu(dz), \quad  y \in \Rdpint.  
\end{align}
%??? return here
%We will study the behavior of $f_\phi(e^{-sy})$ as $s \to \infty$, which corresponds to 
Since $\Phi$ is compactly supported in $\bb R_+^*$ and $\nu$ is a Radon measure, 
$f_\Phi(y)$ is finite for every $y \in \Rdpint$. 
%\todo{This is not ture, e.g., take $y = (0, \ldots, 0, 1)$, the set $\{ z:  \langle y, z \rangle \in (0, 1) \}$ is not compact}
We now introduce the Poisson equation associated with the transform $f_\Phi$. 
For $y \in \Rdpint$, we define the function
\begin{align}\label{Poisson-Psi-Phi}
\Psi_{\Phi}(y) = \bb E f_\Phi \big(A^\top y\big) - f_\Phi(y).
\end{align}
Thus the function $\Psi_\Phi$ measures the defect of $f_\Phi$ from being harmonic
with respect to the transpose action of the random matrices. The
invariance of $\nu$ gives the following useful representation of $\Psi_\Phi$.

\begin{lemma}\label{Lem-formula-Psi-Phi}
Let $\Phi \in C^1_c(\bb R_+^*)$. Then, for any $y \in \Rdpint$, we have 
\begin{align*}%\label{}
\Psi_{\Phi}(y)
=    \int_{\operatorname{Aff}(\bb R_+^d)}  \int_{\bb R^d_+}
\Big( \Phi\big( \langle y, Az \rangle \big) -  \Phi\big( \langle y, Az + B \rangle \big)  \Big)  \nu(dz)  \mu(dA, dB). 
\end{align*}
\end{lemma} 

\begin{proof}
By \eqref{def-function-f-Phi} and the invariance of the measure $\nu$ (cf.\  \eqref{def invariant measure}), 
we get 
\begin{align*}%\label{}
f_\Phi(y) = \int_{\bb R^d_+}  \Phi\big( \langle y, z \rangle \big) \nu(dz)
=  \int_{\bb R^d_+}  \int_{\textrm{Aff}(\bb R_+^d)}  \Phi\big( \langle y, Az + B \rangle \big)  \mu(dA, dB) \nu(dz). 
\end{align*}
Note also that, using Fubini's theorem for nonnegative functions, 
\begin{align*}%\label{}
\bb E f_\Phi \big(A^\top y\big) 
&=  \bb E  \int_{\bb R^d_+}  \Phi\big( \langle A^\top y, z \rangle \big) \nu(dz) 
= \bb E  \int_{\bb R^d_+}  \Phi\big( \langle y, Az \rangle \big) \nu(dz) \\
&=  \int_{\bb R^d_+}  \int_{\textrm{Aff}(\bb R_+^d)}  \Phi\big( \langle y, Az \rangle \big)  \mu(dA, dB) \nu(dz)
\end{align*}
The conclusion follows by \eqref{Poisson-Psi-Phi}. 
\end{proof}

For any $y \in \bb R^d_+ \setminus\{0\}$, 
we denote 
\begin{align*}%\label{}
 y_i = \langle e_i, y \rangle 
\quad \mbox{and} \quad
\overline y = \frac{y}{|y|}. 
\end{align*}
Then, for any $y, z \in \bb R^d_+ \setminus\{0\}$,  
\begin{align}\label{Important-Inequality} 
\langle y, z \rangle \geq  \langle \bf 1, z \rangle  |y|  \min_{1 \leq i \leq d}  \overline y_i
=  |y| |z| \min_{1 \leq i \leq d}  \overline y_i. 
\end{align}
%This elementary inequality will be used repeatedly. 
%Denote by $C^1_c(\bb R_+^*)$ the space of compactly supported and differentiable functions on the interval $E \subseteq \bb R$.  
%hence the integral converges and $f_\Phi$ is finite on $\mathbb R^d$ (under the usual properties of $\nu$). 
Thanks to Lemma \ref{Lem-bound.nu-z.gamma}, 
%the following lemma provides its decay rate at infinity. 
the following estimate provides a quantitative
decay bound as $|y|\to\infty$, uniformly over directions that remain
away from the boundary of the positive cone.

\begin{lemma}\label{Lem-f-Phi-decay-rate-infinity}
Let $\Phi \in C^1_c(\bb R_+^*)$ and let $0< \gamma < \delta/2$, 
where $\delta$ is the exponent in \ref{Condition-moments-A-B}. 
%For any $\gamma \in (0, d)$ and $\Phi \in C^1_c(\bb R_+^*)$, 
Then, there is a constant $c >0$ (depending on $\gamma$ and $\Phi$) such that for all $y \in \Rdpint$, 
\begin{align*}
|f_\Phi(y)| \le c \Big( \min_{1 \leq i \leq d}  \overline y_i \Big)^{-\gamma}  |y|^{-\gamma}.
\end{align*}
\end{lemma}

\begin{proof} 
Since $\Phi$ has compact support contained in $\bb R_+^* = (0, \infty)$, 
for any $\gamma \in (0, \delta/2)$, there exists a constant $c>0$ (depending on $\gamma$ and $\Phi$)
 such that $|\Phi(t)| \leq c  t^{-\gamma}$ for all $t \in \bb R_+^*$.
Therefore, by \eqref{def-function-f-Phi} and \eqref{Important-Inequality}, 
there exists a constant $c' >0$ such that for any $y \in \Rdpint$, 
\begin{align*}%\label{}
|f_\Phi(y)| \leq  c  \int_{\bb R^d_+}  \frac{1}{\langle y, z \rangle^{\gamma}} \nu(dz)
\leq   \frac{c}{|y|^{\gamma}}  \Big( \min_{1 \leq i \leq d}  \overline y_i \Big)^{-\gamma}  \int_{\bb R^d_+}  \frac{1}{|z|^{\gamma}} \nu(dz) 
\leq \frac{c'}{|y|^{\gamma}}  \Big( \min_{1 \leq i \leq d}  \overline y_i \Big)^{-\gamma}, 
\end{align*}
where the finiteness of the integral is due to Lemma \ref{Lem-bound.nu-z.gamma}. 
\end{proof}

The following result shows that, under condition \ref{Condition-Furstenberg-Kesten}, the Markov chain $(X_n^x)_{n \geq 0}$ enters, 
after one step, a fixed compact subset of the interior of the positive sphere. 
For a matrix $g \in \bb{M}$, denote $g \cdot \bb S_+^{d-1} = \{ g \cdot x:  x \in  \bb S_+^{d-1}  \}.$ Further, abbreviate
$$\bb S^{d-1}_{+, \epsilon}
= \{ x \in \bb S_+^{d-1}:  \langle y, x \rangle \geq \epsilon \ \mbox{for all $y \in \bb S_+^{d-1}$} \}.$$

\begin{lemma} \label{lem equiv Kesten}
%Assume condition \ref{Condition-Furstenberg-Kesten}. 
%Let $A \cdot \bb S_+^{d-1} = \{ A \cdot x:  x \in  \bb S_+^{d-1}  \}.$
With $\epsilon =  \frac{1}{\varkappa d} \in (0, 1)$, we have 
\begin{align}\label{Inequality strict positivity}
A \cdot \bb S_+^{d-1} \subseteq  \bb S^{d-1}_{+, \epsilon}
\quad  \mbox{and} \quad  A^\top \cdot \bb S_+^{d-1} \subseteq  \bb S^{d-1}_{+, \epsilon}
\quad  \mbox{for $A \in \supp \mu$}. 
\end{align}
%where $\bb S^{d-1}_{+, \epsilon}
%= \{ x \in \bb S_+^{d-1}:  \langle y, x \rangle \geq \epsilon \ \mbox{for all $y \in \bb S_+^{d-1}$} \}.$
In particular, for all $x \in \bb S_+^{d-1}$ and $n \geq 1$, 
we have $\min_{1 \leq i \leq d} \langle e_i, X_n^x \rangle \geq \epsilon$, $\bb P$-almost surely. 
\end{lemma}

%\begin{lemma}\label{Lem-integrability-min}
%Under condition \ref{Condition-Furstenberg-Kesten}, 
%for all $x \in \bb S_+^{d-1}$, $n \geq 1$ and $\gamma >0$, we have 
%\begin{align*}%\label{}
%\bb E \left[ \min_{1 \leq i \leq d} \langle e_i, X_n^x \rangle^{-\gamma} \right]  \leq  \varkappa^{\gamma} d^{1 + \gamma}, 
%\end{align*}
%where $\varkappa$ is from \ref{Condition-Furstenberg-Kesten}. 
%\end{lemma}

\begin{proof}
The proof of \eqref{Inequality strict positivity} can be found in \cite[Lemma 2.2]{XGL25}. 
%Note that 
%\begin{align*}
%\bb E \left[ \min_{1 \leq i \leq d} \langle e_i, X_n^x \rangle^{-\gamma} \right]
%\leq \sum_{i = 1}^d  \bb E  \big(  \langle e_i, X_n^x \rangle^{-\gamma} \big). 
%\end{align*}
%By Lemma \ref{lem equiv Kesten}, with $\epsilon =  \frac{1}{\varkappa d} \in (0, 1)$, 
%we have $\bb E  \big(  \langle e_i, X_n^x \rangle^{-\gamma} \big)
% \leq  (\varkappa d)^{\gamma}.$
%%\begin{align*}%\label{}
%%\bb E  \big(  \langle e_i, X_n^x \rangle^{-\gamma} \big)
%% \leq  (\varkappa d)^{\gamma}. 
%%\end{align*}
%The assertion follows. 
\end{proof}

As a first application of the previous lemmata, we derive quantitative bounds for the function $\Psi_\Phi$. 
The first estimate controls its behavior at infinity.

\begin{lemma}\label{Lem-upper-bound-Psi-Phi-y-minus}
Let $\Phi \in C^1_c(\bb R_+^*)$ and let $0< \gamma < \delta/2$, 
where $\delta$ is the exponent in \ref{Condition-moments-A-B}. 
Then, there exists a constant $c >0$ (depending on $\gamma$ and $\Phi$) such that for any $y \in \Rdpint$,  
\begin{align*}
|\Psi_{\Phi}(y)|  
\leq c |y|^{-\gamma}. 
\end{align*}
\end{lemma}

\begin{proof}
Since $\Phi$ has compact support contained in $\bb R_+^*$, 
there is a constant $c >0$ (depending on $\gamma$ and $\Phi$) such that
$|\Phi(u)| \leq c u^{-\gamma}$ for all $u \in \bb R_+^*$. 
Using Lemma \ref{Lem-formula-Psi-Phi}, the nonnegativity of the coordinates of $B$, and
the triangle inequality, we obtain
\begin{align*}
 |\Psi_\Phi(y)| & \leq  c  \int_{\operatorname{Aff}(\mathbb R_+^d)} \int_{\mathbb R_+^d}
 \left[  \langle y,Az\rangle^{-\gamma}  +  \langle y,Az+B\rangle^{-\gamma}  \right]  \nu(dz) \, \mu(dA,dB)  \\
 & \leq  2c  \int_{\operatorname{Aff}(\bb R_+^d)}  \int_{\mathbb R_+^d}
 \langle y, Az \rangle^{-\gamma}  \nu(dz) \, \mu(dA,dB).
\end{align*}
By Lemma \ref{lem equiv Kesten}, we have $\min_{1\le i\le d} ( \overline{Az})_i \geq \epsilon$ for any $z \in \bb R_+^d \setminus \{0\}$, 
and therefore \eqref{Important-Inequality} gives
\begin{align*}%\label{}
\langle y,  Az \rangle \geq  \epsilon |y| \, |Az|  \geq \epsilon |y| \,\iota(A)|z|, 
\end{align*}
where $\iota(A) = \inf_{z \in \bb R^d_+ \setminus \{0\} } \frac{|Az|}{|z|}$. 
It follows that
\begin{align*}%\label{}
|\Psi_\Phi(y)| \leq  c' |y|^{-\gamma}  \bb E \bigl[\iota(A)^{-\gamma}\bigr] \int_{\bb R_+^d}|z|^{-\gamma}\,\nu(dz).
\end{align*} 
Both factors on the right-hand side are finite: the first one due to \ref{Condition-moments-A-B}, 
since $\iota(A)^{-1}\leq N(A)$, and the second one
due to Lemma \ref{Lem-bound.nu-z.gamma}. This proves the lemma.
\end{proof}

We can further provide a refined bound for $\Psi_{\Phi}(y)$. 
Besides being useful in its own right, this estimate implies the direct Riemann integrability required in the renewal argument below. 
For a continuous function $h$ on $\Rdpint$, 
we say that $h$ is directly Riemann integrable if 
\begin{align}\label{def-directly Riemann integrable}
\sum_{k \in \bb Z} \sup_{r \in [e^k, e^{k+1}) } \sup_{x \in (\bb S^{d-1}_+)^\circ}  |h( r x )| < \infty. 
\end{align}
In the same way, for a continuous function $\tilde{h}$ on $(\bb S^{d-1}_+)^\circ \times \bb R$, 
we say that $\tilde{h}$ is directly Riemann integrable if
\begin{align}\label{def-directly Riemann integrable2}
	\sum_{k \in \bb Z} \sup_{r \in [k, {k+1}) } \sup_{x \in (\bb S^{d-1}_+)^\circ}  |\tilde{h}(x,r )| < \infty. 
\end{align}

\begin{lemma}\label{Lem-integrability}
Let $\Phi \in C^1_c(\bb R_+^*)$ and let $0< \gamma < \delta/2$, 
where $\delta$ is the exponent in \ref{Condition-moments-A-B}.  
%For any $\gamma \in (0, d)$ and  $\Phi \in C^1_c(\bb R_+^*)$, 
Then, there exists a constant $c>0$ such that for all $y \in \Rdpint$, 
\begin{align}\label{Bound-Psi-Phi-min}
|\Psi_\Phi(y)|  \leq c \min \left\{ |y|^{\delta/4}, |y|^{-\gamma} \right\}. 
\end{align}
%where $\delta > 0$ is from condition \ref{Condition-moments-A-B}. 
In particular,  the function $y \mapsto \Psi_{\Phi}(y)$ 
is continuous and directly Riemann integrable on $\Rdpint$.
\end{lemma}

\begin{proof}
The estimate at infinity follows from Lemma \ref{Lem-upper-bound-Psi-Phi-y-minus}. 
It therefore remains to control $\Psi_\Phi(y)$ as $|y|\to0$.
Assume that $\Phi$ is supported on $[a, b]$ for some constants $0<a<b<\infty$. 
By Lemma \ref{Lem-formula-Psi-Phi}, we have 
\begin{align}\label{Decomp-Psi-I1y-I2y}
|\Psi_{\Phi}(y)|  
& \leq   \int_{\textrm{Aff}(\bb R_+^d)}   \int_{\bb R^d_+} 
\Big| \Phi\big( \langle y, Az + B \rangle \big) -  \Phi\big( \langle y, Az \rangle \big)   \Big|  \nu(dz)  \mu(dA, dB)  \notag\\
& \leq  \int_{\{(A, B) \in \textrm{Aff}(\bb R_+^d):  \, |y||B| \leq \frac{a}{2} \}}   \int_{\bb R^d_+} 
\Big| \Phi\big( \langle y, Az + B \rangle \big) -  \Phi\big( \langle y, Az \rangle \big)   \Big|  \nu(dz)  \mu(dA, dB)  \notag\\
& \quad  +  \int_{\{(A, B) \in \textrm{Aff}(\bb R_+^d): \,  |y||B| > \frac{a}{2} \}}   \int_{\bb R^d_+} 
\Big| \Phi\big( \langle y, Az + B \rangle \big) -  \Phi\big( \langle y, Az \rangle \big)   \Big|  \nu(dz)  \mu(dA, dB) \notag\\
& = I_1(y) + I_2(y). 
\end{align}

\textbf{Step 1}. We first deal with $I_1(y)$. 
Since $\Phi$ is continuously differentiable on $[a, b]$, 
we can set $c_1 = \max_{x \in [a, b]} \abs{\Phi'(x)} < \infty$.
Then, for every $|w| \leq  \frac{a}{2}$, it holds that for any $r \in \bb R_+$, 
\begin{align*}
\abs{\Phi(r+w)-\Phi(r) } \le c_1 |w|  \ds{1}_{[\frac{a}{2}, b + \frac{a}{2}]}(r). 
\end{align*}
It follows that, on the set $\{ |y||B| \leq \frac{a}{2} \}$, 
\begin{align}\label{inequa-Phi-diffe-001}
\Big| \Phi\big( \langle y, Az + B \rangle \big) -  \Phi\big( \langle y, Az \rangle \big)   \Big|
& \leq  c_1 \langle y, B \rangle \ds{1}_{[\frac{a}{2}, b + \frac{a}{2}]}(\langle y, Az \rangle)  \notag\\
& \leq c_2 |y| |B| \langle y, Az \rangle^{-\eta}  \notag\\
& \leq  c_2  |y| |B| |y|^{-\eta} |Az|^{-\eta} \Big(\min_{1 \leq i \leq d} (\overline{Az})_i  \Big)^{-\eta}  \notag\\
& \leq  c_2  |y| |B| |y|^{-\eta}  \iota(A)^{-\eta}  |z|^{-\eta}  \Big(\min_{1 \leq i \leq d} (\overline{Az})_i  \Big)^{-\eta}, 
\end{align}
where in the second inequality we used $\ds{1}_{[a',b']}(u) \le  (b')^{\eta} u^{-\eta}$,
in the third one we used \eqref{Important-Inequality},
and in the last one we used $|Az|^{-\eta} \leq \iota(A)^{-\eta} |z|^{-\eta}$, 
with $\iota(A) = \inf_{v \in \bb R^d_+ \setminus \{0\} } \frac{|Av|}{|v|}$. 

By condition \ref{Condition-moments-A-B}
(replacing $\delta$ by $\min\{\delta,1\}$ if necessary, we may assume $0 < \delta \leq 1$), 
there is a constant $0< \delta < 1$
such that $\bb E \big[ N(A)^{\delta}\big]<\infty$ and $\bb E (|B|^{\delta}) <~\infty$. 
Notice that, on the set $\{ |y||B| \leq \frac{a}{2} \}$, 
it holds that $$|y| |B| = (|y| |B|)^{1 - \frac{\delta}{2}} (|y| |B|)^{\frac{\delta}{2}}
 \leq \left( \frac{a}{2} \right)^{1 - \frac{\delta}{2}} |y|^{\frac{\delta}{2}} |B|^{\frac{\delta}{2}}.$$
Hence, we get 
\begin{align*}%\label{}
|I_1(y)|  \leq  c_2  \left( \frac{a}{2} \right)^{1 - \frac{\delta}{2}} |y|^{ \frac{\delta}{2} - \eta}  \int_{ \textrm{Aff}(\bb R_+^d) }  
 \int_{\bb R^d_+}  |B|^{\frac{\delta}{2}}  \iota(A)^{-\eta}  |z|^{-\eta}  \Big(\min_{1 \leq i \leq d} (\overline{Az})_i \Big)^{-\eta} \nu(dz)  \mu(dA, dB). 
\end{align*}
Using H\"older's inequality, condition \ref{Condition-moments-A-B} and Lemma \ref{lem equiv Kesten}, 
and taking $\eta \in (0, \delta/4)$, we obtain 
\begin{align}\label{sup-inequa-bound-001}
 & \sup_{z \in \bb R^d_+ \setminus \{0\}} \int_{ \textrm{Aff}(\bb R_+^d) }  
  |B|^{\frac{\delta}{2}}  \iota(A)^{-\eta}    \Big(\min_{1 \leq i \leq d} (\overline{Az})_i \Big)^{-\eta}  \mu(dA, dB)  \notag\\
& \leq  \left[  \bb E (|B|^{\delta}) \right]^{1/2}
  \left[ \bb E \left( \iota(A)^{- 4 \eta} \right)  \right]^{1/4} 
 \sup_{z \in \bb R^d_+ \setminus \{0\}}  \left[  \bb E  \left( \min_{1 \leq i \leq d} (\overline{Az})_i \right)^{-4 \eta}  \right]^{1/4}  
  \leq  c < \infty. 
\end{align}
Therefore, by Fubini's theorem and Lemma \ref{Lem-bound.nu-z.gamma}, we get
\begin{align}\label{estimate-I1y}
|I_1(y)| \leq  c'  |y|^{ \frac{\delta}{2} - \eta} \int_{\bb R^d_+}  |z|^{- \eta}  \nu(dz)  
\leq  c''  |y|^{ \frac{\delta}{2} - \eta} 
\leq  c'' |y|^{ \frac{\delta}{4}},  
\end{align}
where the last inequality holds since $\eta \in (0, \delta/4)$. 

\textbf{Step 2}. Next we deal with the second term $I_2(y)$. 
We have 
\begin{align*}%\label{}
I_2(y) & \leq 
 \int_{ \left\{(A, B) \in \textrm{Aff}(\bb R_+^d): \, |y||B| > \frac{a}{2} \right\}}   \int_{\bb R^d_+} 
\Big|  \Phi\big( \langle y, Az \rangle \big)   \Big|  \nu(dz)  \mu(dA, dB)  \notag\\
& \quad +  \int_{ \left\{(A, B) \in \textrm{Aff}(\bb R_+^d):  \,  |y||B| > \frac{a}{2} \right\}}   \int_{\bb R^d_+} 
\Big| \Phi\big( \langle y, Az + B \rangle \big)   \Big|  \nu(dz)  \mu(dA, dB)  \notag\\
& = I_{21}(y) + I_{22}(y). 
\end{align*}
\textbf{Step 2a}. For $I_{21}(y)$, since $\Phi$ is supported on $[a, b]$ for some constants $0<a<b<\infty$,
similarly to the proof of \eqref{inequa-Phi-diffe-001}, by \eqref{Important-Inequality} we have 
\begin{align}\label{estimate-I21-a}
\Big|  \Phi\big( \langle y, Az \rangle \big)   \Big|
& \leq  \| \Phi \|_{\infty}  \ds{1}_{[a, b]}(\langle y, Az \rangle)   \notag\\
& \leq c \langle y, Az \rangle^{-\eta}  \notag\\
& \leq c |y|^{-\eta} |Az|^{-\eta} \Big(\min_{1 \leq i \leq d} (\overline{Az})_i \Big)^{-\eta}  \notag\\
& \leq c |y|^{-\eta}  \iota(A)^{-\eta}  |z|^{-\eta} \Big(\min_{1 \leq i \leq d} (\overline{Az})_i \Big)^{-\eta}. 
\end{align}
Hence, using the fact that $|y||B| > \frac{a}{2}$, Fubini's theorem, \eqref{sup-inequa-bound-001} and Lemma \ref{Lem-bound.nu-z.gamma}, 
we see that 
\begin{align}\label{estimate-I21-b}
I_{21}(y) & \leq  c' |y|^{-\eta}  \int_{ \left\{(A, B) \in \textrm{Aff}(\bb R_+^d):  \,  |y||B| > \frac{a}{2}  \right\}}   \int_{\bb R^d_+} 
\iota(A)^{-\eta}  |z|^{-\eta}  \Big(\min_{1 \leq i \leq d} (\overline{Az})_i \Big)^{-\eta}   \nu(dz)  \mu(dA, dB)  \notag\\
& \leq  c' \left( \frac{2}{a} \right)^{ \frac{\delta}{2} } |y|^{\frac{\delta}{2}-\eta}   \int_{\bb R^d_+}  \int_{ \textrm{Aff}(\bb R_+^d) }  
 |B|^{\frac{\delta}{2}}  \iota(A)^{-\eta}  |z|^{-\eta}  \Big(\min_{1 \leq i \leq d} (\overline{Az})_i \Big)^{-\eta}     \mu(dA, dB) \nu(dz)  \notag\\
& \leq c'' |y|^{\frac{\delta}{2}- \eta}  \int_{\bb R^d_+}  |z|^{- \eta} \nu(dz)  \notag\\
& \leq c''' |y|^{\frac{\delta}{4}}, 
\end{align}
where in the last inequality we used $\eta \in (0, \delta/4)$. 

\textbf{Step 2b}. For $I_{22}(y)$, taking into account that both $y$ and $B$ are in $\bb R_+^d$, we have 
\begin{align*}%\label{}
\Big| \Phi\big( \langle y, Az + B \rangle \big)   \Big|
\leq  \| \Phi \|_{\infty}  \ds{1}_{[a, b]}(\langle y, Az + B \rangle)  
\leq  \| \Phi \|_{\infty}  \ds{1}_{[0, b]}(\langle y, Az \rangle). 
\end{align*}
In the same way as in the proof of \eqref{estimate-I21-a}, it follows that
\begin{align}\label{estimate-I21-aa}
\ds{1}_{[0, b]}(\langle y, Az \rangle)
\leq  c \langle y, Az \rangle^{-\eta} 
\leq c |y|^{-\eta}  \iota(A)^{-\eta}  |z|^{-\eta}  \Big(\min_{1 \leq i \leq d} (\overline{Az})_i \Big)^{-\eta}. 
\end{align}
Hence, by \eqref{estimate-I21-b}, we obtain 
\begin{align}\label{estimate-I21-c}
I_{22}(y)  \leq c |y|^{\frac{\delta}{2} - \eta}  \int_{\bb R^d_+}  |z|^{-\eta} \nu(dz) 
\leq c' |y|^{\frac{\delta}{2}-\eta} 
\leq c' |y|^{\frac{\delta}{4}}. 
\end{align}
Substituting \eqref{estimate-I1y}, \eqref{estimate-I21-b} and \eqref{estimate-I21-c} into \eqref{Decomp-Psi-I1y-I2y},
we get that there exists a constant $c>0$ such that for any $y \in \Rdpint$, 
\begin{align*}%\label{}
|\Psi_{\Phi}(y)| \leq c |y|^{\frac{\delta}{4}}. 
\end{align*}
Together with Lemma \ref{Lem-upper-bound-Psi-Phi-y-minus}, this completes the proof of \eqref{Bound-Psi-Phi-min}. 

By the definition \eqref{def-directly Riemann integrable} and using \eqref{Bound-Psi-Phi-min}, 
we get that the function $y \mapsto \Psi_{\Phi}(y)$
is directly Riemann integrable on $\Rdpint$. 

The continuity of $\Psi_{\Phi}$ follows from the representation in Lemma \ref{Lem-formula-Psi-Phi}
and the fact that we may use dominated convergence to study $\lim_{n \to \infty} \Psi_{\Phi}(y_n)$ due to the bounds 
\eqref{inequa-Phi-diffe-001}, \eqref{estimate-I21-a} and \eqref{estimate-I21-aa}. 
\end{proof}

The next lemma shows the cancellation property of the function $\Psi_\Phi$. 
This property will allow the renewal problem to be reduced to the study of the function $\overline\Psi_\Phi$ introduced below. 

\begin{lemma}\label{Lem-integrability002}
For any $\Phi \in C^1_c(\bb R_+^*)$ and $x \in (\bb S_+^{d-1})^\circ$, we have 
\begin{align}\label{symmetry identity}
\int_{-\infty}^{\infty}  \Psi_\Phi(e^{-s} x) ds = 0. 
\end{align}
Moreover, the function 
\begin{align*}%\label{}
\overline \Psi_\Phi(x, t) = \int_{-\infty}^{t} \Psi_\Phi(e^{-s} x) ds
\end{align*}
is continuous and directly Riemann integrable on $(\bb S_+^{d-1})^\circ \times \bb R$. 
\end{lemma} 

\begin{proof}
By the direct Riemann integrability of $\Psi$, proved in Lemma \ref{Lem-integrability}, and its representation, proved in Lemma \ref{Lem-formula-Psi-Phi}, 
it holds that, for any $x \in (\bb S_+^{d-1})^\circ$, 
\begin{align}\label{eq:bound integral Phi}
\int_{-\infty}^{\infty} \int_{\textrm{Aff}(\bb R_+^d)}   \int_{\bb R^d_+} 
\Big|  \Phi\big( \langle e^{-s} x, Az \rangle \big) - \Phi\big( \langle e^{-s} x, Az + B \rangle \big)   \Big|  \nu(dz)  \mu(dA, dB) ds < \infty. 
\end{align}
Therefore, we may apply Fubini's theorem to obtain that, for any $x \in (\bb S_+^{d-1})^\circ$, 
\begin{align*}%\label{}
& \int_{-\infty}^{\infty}  \Psi_\Phi(e^{-s} x) ds \notag\\
& =  \int_{-\infty}^{\infty} \int_{\textrm{Aff}(\bb R_+^d)}   \int_{\bb R^d_+} 
\Big( \Phi\big( \langle e^{-s} x, Az \rangle \big) - \Phi\big( \langle e^{-s} x, Az + B \rangle \big)   \Big)  \nu(dz)  \mu(dA, dB) ds \notag\\
& = \int_{\textrm{Aff}(\bb R_+^d)}   \int_{\bb R^d_+}   \int_{-\infty}^{\infty}
\Big( \Phi\big(  e^{-s + \log \langle  x, Az \rangle } \big) -  \Phi\big( e^{-s + \log \langle  x, Az + B \rangle } \big)    \Big)  ds \nu(dz)  \mu(dA, dB) 
\notag\\
& = 0. 
\end{align*}
Here we have used the translation invariance of the Lebesgue measure to arrive at the proof of  \eqref{symmetry identity}. 
The bound \eqref{eq:bound integral Phi} also allows us to use the dominated convergence theorem to infer the continuity of $\overline{\Psi}$ from the continuity of $\Psi$.

Using \eqref{Bound-Psi-Phi-min}, 
we get that, for any $t \leq 0$ and any $\gamma \in (0, \delta/2)$
with $\delta > 0$ from condition \ref{Condition-moments-A-B}, we have 
\begin{align*}%\label{}
|\overline \Psi_\Phi(x, t)| \leq  \int_{-\infty}^{t} |\Psi_\Phi(e^{-s} x)| ds 
\leq  c  \int_{-\infty}^{t}  e^{s \gamma} ds  = \frac{c}{\gamma}e^{- \gamma |t|}. 
\end{align*}
For $t>0$, from \eqref{symmetry identity} we get that 
\begin{align*}%\label{}
|\overline \Psi_\Phi(x, t)| = \left| \int_{-\infty}^{t} \Psi_\Phi(e^{-s} x) ds \right| 
=  \left| \int_{t}^{\infty} \Psi_\Phi(e^{-s} x) ds \right|
\leq   \int_{t}^{\infty}  \left|  \Psi_\Phi(e^{-s} x) \right| ds. 
\end{align*}
Using again \eqref{Bound-Psi-Phi-min}, we get 
\begin{align*}%\label{}
\int_{t}^{\infty}  \left|  \Psi_\Phi(e^{-s} x) \right| ds 
\leq c  \int_{t}^{\infty} e^{-s \delta/4} ds  = \frac{4c}{\delta}  e^{- t \delta /4}. 
\end{align*}
%where $\delta > 0$ is from condition \ref{Condition-moments-A-B}. 
Therefore, the function $\overline \Psi_\Phi$ is directly Riemann integrable on $(\bb S_+^{d-1})^\circ \times \bb R$. 
\end{proof}

Taking $x \in (\bb S_+^{d-1})^\circ$ and $y=e^{-s}x$ with $s \in \bb R$ in the Poisson equation \eqref{Poisson-Psi-Phi}, 
we obtain 
\begin{align*}
\Psi_\Phi(e^{-s}x) 
= \bb E  f_\Phi \left( e^{\log |A_1^\top x| - s}  (A_1^\top \cdot x)  \right) - f_\Phi(e^{-s}x) 
= \bb E  f_\Phi \left(e^{S_1^x-s}X_1^x \right) - f_\Phi(e^{-s}x), 
\end{align*}
where $X_1^x$ and $S_1^x$ are defined in \eqref{def-Xn-Sn}. 
This leads to the following martingale associated with the Poisson equation. 
Define
\begin{align*}%\label{}
M_0(e^{-s}x) = f_\Phi (e^{-s} x)
\end{align*}
and for $n \geq 1$,
\begin{align}\label{def martingale Mn}
M_n(e^{-s}x) = f_\Phi \Big(e^{S_n^x-s} X_n^x \Big) - \sum_{k=0}^{n-1} \Psi_{\Phi} \Big( e^{S_k^x-s}X_k^x \Big). 
\end{align}

\begin{lemma}\label{Lem-martingale}
Let $\Phi \in C^1_c(\bb R_+^*)$, $s \in \bb R$ and $x \in (\bb S_+^{d-1})^\circ$. 
Then $M_n(e^{-s}x)$ is a martingale with respect to the natural filtration $(\mathscr F_n)_{n \geq 0}$ 
of $(A_n)_{n\geq 1}$, with $\mathscr F_0$ being the trivial $\sigma$-algebra.
\end{lemma}

\begin{proof}
We first verify integrability. 
By Lemma \ref{Lem-f-Phi-decay-rate-infinity}, we have that, for any $\gamma \in (0, \delta/2)$,
with $\delta$ from condition \ref{Condition-moments-A-B}, 
and for any function $\Phi$ in $C^1_c(\bb R_+^*)$, 
there is a constant $c >0$ (depending on $\gamma$ and $\Phi$) such that for any $n \geq 1$, $s \in \bb R$ and $x \in (\bb S_+^{d-1})^\circ$, 
\begin{align*}%\label{}
\bb E \left| f_\Phi  \Big(e^{S_n^x-s} X_n^x \Big) \right|
 \leq c e^{\gamma s} \bb E \Big( \Big(\min_{1 \leq i \leq d} (X_n^x)_i  \Big)^{-\gamma}  e^{- \gamma S_n^x} \Big). 
\end{align*}
Using the Cauchy-Schwarz inequality and Lemma \ref{lem equiv Kesten}, we get 
\begin{align*}%\label{}
\bb E  \Big(  \Big(\min_{1 \leq i \leq d} (X_n^x)_i  \Big)^{-\gamma}  e^{- \gamma S_n^x}  \Big)
\leq  \bb E^{1/2} \Big(  \Big(\min_{1 \leq i \leq d} (X_n^x)_i  \Big)^{-2\gamma}  \Big)  \bb E^{1/2} \Big( e^{- 2\gamma S_n^x} \Big)
\leq  c \Big[ \bb E \iota(A_1^\top)^{-2\gamma} \Big]^{n/2}. 
% < \infty,  
\end{align*}
By \eqref{def-norm-g-iota} and condition \ref{Condition-Furstenberg-Kesten}, 
we know that $\bb E \iota(A_1^\top)^{-2\gamma} \leq c \bb E \iota(A_1)^{-2\gamma}$. 
Hence 
\begin{align*}%\label{}
\bb E \left| f_\Phi \Big(e^{S_n^x-s} X_n^x \Big) \right|
 \leq c e^{\gamma s} \left[ \bb E \iota(A_1)^{-2\gamma} \right]^{n/2} < \infty, 
\end{align*}
where we take $\gamma \in (0, \delta/2)$ with $\delta>0$ from condition \ref{Condition-moments-A-B}. 
Similarly, using Lemma \ref{Lem-integrability}, we get that for any $n \geq 1$, 
\begin{align*}%\label{}
\bb E \left| \Psi_{\Phi} \Big( e^{S_n^x-s}X_n^x \Big) \right| 
\leq  c e^{\gamma s}  \bb E  e^{-\gamma S_n^x}  \leq c e^{\gamma s} \left[ \bb E \iota(A_1^\top)^{-\gamma} \right]^n
< \infty. 
\end{align*}
Therefore, for any $n \geq 1$, $s \in \bb R$ and $x \in (\bb S_+^{d-1})^\circ$, we have $\bb E |M_n(e^{-s}x)| < \infty$, 
so we may take expectations and conditional expectations to verify the martingale property below. 

Note that 
\begin{align}\label{maringale-property-001}
\bb E \Big[ M_n(e^{-s}x)  \, \Big|\, \mathscr F_{n-1} \Big]
 = \bb E \Big[ f_\Phi \Big( e^{S_n^x-s} X_n^x \Big) \, \Big| \, \mathscr F_{n-1}\Big] 
  - \sum_{k=0}^{n-1} \Psi_{\Phi} \Big( e^{S_k^x-s}X_k^x \Big). 
\end{align}
By \eqref{def-Xn-Sn} and \eqref{Poisson-Psi-Phi}, we get 
\begin{align*}
&  \bb E \Big[ f_\Phi  \Big(e^{S_n^x-s} X_n^x \Big) \, \Big| \, \mathscr F_{n-1}\Big]  \\
& = \bb E \Big[ f_\Phi  \Big(e^{S_n^x-s} X_n^x \Big) - f_\Phi \Big(e^{S_{n-1}^x-s} X_{n-1}^x \Big) \, \Big| \, \mathscr F_{n-1}\Big] 
 + f_\Phi  \Big(e^{S_{n-1}^x-s}X_{n-1}^x \Big)  \\
& = \bb E \Big[ f_\Phi  \Big(A_n^\top e^{S_{n-1}^x-s}X_{n-1}^x \Big) 
 - f_\Phi  \Big( e^{S_{n-1}^x-s}X_{n-1}^x \Big) \, \Big| \, \mathscr F_{n-1}\Big] + f_\Phi  \Big( e^{S_{n-1}^x-s}X_{n-1}^x \Big)  \\
& = \Psi_\Phi  \Big( e^{S_{n-1}^x-s}X_{n-1}^x \Big) + f_\Phi \Big( e^{S_{n-1}^x-s}X_{n-1}^x \Big). 
\end{align*}
Substituting this into \eqref{maringale-property-001} gives 
\begin{align*}%\label{}
\bb E \Big[ M_n(e^{-s}x)  \, \Big|\, \mathscr F_{n-1} \Big]
& = \Psi_\Phi \Big( e^{S_{n-1}^x-s}X_{n-1}^x \Big) + f_\Phi \Big( e^{S_{n-1}^x-s}X_{n-1}^x \Big) 
 -  \sum_{k=0}^{n-1} \Psi_{\Phi} \Big( e^{S_k^x-s}X_k^x \Big)  \notag\\
& = f_\Phi  \Big( e^{S_{n-1}^x-s}X_{n-1}^x \Big)  - \sum_{k=0}^{n-2} \Psi_{\Phi} \Big( e^{S_k^x-s}X_k^x \Big) \notag\\
& = M_{n-1}(e^{-s}x), 
\end{align*}
which proves the assertion. 
\end{proof}

\subsection{Renewal estimates for the centered Markov random walk}\label{Subsection renewal}

In this subsection we collect several estimates for the centered Markov random walk $(X_n^x, S_n^x)_{n \geq 0}$ 
that will be used in the proof of the main results. 
We first recall a conditioned local limit estimate for products of positive random matrices 
and derive a quantitative bound on the corresponding descending ladder epochs. 
We then recall the killed renewal estimate from \cite{GMX25}  
and deduce exponential integrability of the overshoots. 

%We finish this section by quoting from \cite{GMX25} the renewal theoretic result for centered Markov random walks 
%that will be crucial for our proof.
For any $x \in \bb S^{d-1}_+$ and $b \in \bb R$, define 
\begin{align*}%\label{}
\tau_{x, b} = \inf\{ k \ge 1 \, : \,  b + S_k^x < 0\}
\end{align*}
with the convention $\inf \emptyset = \infty$. 
When $b = 0$, we simply write $\tau_x = \tau_{x, 0}$. 
In the proof of the main results (see Section \ref{Section-proof} below), it will also be useful to introduce the stopping time
\begin{align}\label{def stopping T}
T = \inf\{ k \ge 1 \, : \, \log\| A_k^\top \cdots A_1^\top \| < 0\}. 
\end{align}
By \ref{Condition-Furstenberg-Kesten} and Lemma 4.1 in \cite{GX23}, 
with $\varkappa > 1$ as in \ref{Condition-Furstenberg-Kesten}, 
we have 
\begin{align*}
|A^\top x|  \leq  \| A^\top \|  \leq \varkappa^2  |A^\top x| 
\end{align*}
for any $x \in \bb S^{d-1}_+$ and $A \in \Gamma$, 
so that $\log |A^\top x| \leq \log \|A^\top \| \leq 2 \log \varkappa + \log |A^\top x|$. 
Therefore, it follows that, for any $x \in \bb S^{d-1}_+$, $\bb P$-almost surely, 
\begin{align}\label{important inequality}
\tau_{x} \leq T \leq \tau_{x, 2 \log \varkappa}. 
\end{align}
We will use the following conditioned local limit theorem. % for products of positive random matrices. 

\begin{lemma}[\cite{GX23}] \label{Lem-CLLT-bound}
Assume \ref{proximality}-\ref{Condition-Furstenberg-Kesten}. 
Then there exists a constant $c>0$ such that for any $x \in \bb S_+^{d-1}$, $b \geq 0$, $z \geq 0$, $t \in \bb R$ and $n \geq 1$, 
\begin{align*}
 \bb P \left( b + S_n^x \in  t +  [0, z], \tau_{x, b} > n \right)
\leq  \frac{c}{n^{3/2}} (1 + b)  (1 + z) (1 + z + \max\{t, 0\})
\end{align*}
\end{lemma}

The preceding bound implies a useful estimate for the distribution of the stopping times. 
%The following lemma can be deduced from Lemma \ref{Lem-CLLT-bound}. 

\begin{lemma}\label{Lem-exit-time-estimate}
Assume \ref{proximality}-\ref{Condition-Furstenberg-Kesten}. 
Then, for any $b \geq 0$, there is a constant $c = c(b) >0$ such that for all $x \in \bb S^{d-1}_+$ and $n \geq 1$, 
\begin{align}\label{bound tau x}
\bb P (\tau_{x, b} = n) \leq c n^{- 3/2}. 
\end{align}
Moreover, there is a constant $c >0$ such that for all $n \geq 1$, 
\begin{align}\label{bound T}
\bb P (T = n) \leq c n^{- 3/2}. 
\end{align}
\end{lemma}

\begin{proof}
We only prove \eqref{bound T} since the proof of \eqref{bound tau x} is similar. 
The case $n = 1$ is immediate after increasing the constant, so assume $n \geq 2$. 
Note that 
\begin{align*}
\bb P (T = n)  & = \bb P \left( \log\| A_n^\top  \cdots A_1^\top \| < 0, T > n-1 \right)  \notag\\
& \leq  \bb P \left( 2 \log \varkappa + \log| A_n^\top  \cdots A_1^\top  x | < 2 \log \varkappa, \tau_{x, 2 \log \varkappa} > n-1 \right). 
\end{align*}
Using first the Markov property and then Markov's inequality together with Lemma \ref{Lem-CLLT-bound}, we get 
\begin{align*}
& \bb P \left( 2 \log \varkappa +  \log| A_n^\top  \cdots A_1^\top  x | < 2 \log \varkappa, \tau_{x, 2 \log \varkappa} > n-1 \right) \notag\\
& =  \int_0^{\infty}  \int_{\bb S_+^{d-1}} \bb P \left( t + \log|A_n^\top  x'| < 2 \log \varkappa \right)   \notag\\
& \qquad \qquad\quad
 \times \bb P \left( X_{n-1}^x \in dx',  2 \log \varkappa + \log| A_{n-1}^\top  \cdots A_1^\top  x | \in dt,  \tau_{x, 2 \log \varkappa} > n-1  \right)  \notag\\
& \leq  c \bb E (\log N(A_n)^{2 + \delta})   \int_0^{\infty}  \frac{1}{(1 + t)^{2 + \delta}}
 \bb P \left( 2 \log \varkappa + \log| A_{n-1}^\top  \cdots A_1^\top  x | \in dt,  \tau_{x, 2 \log \varkappa} > n-1  \right)  \notag\\
& \leq  c'   \sum_{j = 0}^{\infty}  \frac{1}{(1 + j)^{2 + \delta}}
 \bb P \left( 2 \log \varkappa + \log| A_{n-1}^\top  \cdots A_1^\top  x | \in [j, j+1),  \tau_{x, 2 \log \varkappa} > n-1  \right)  \notag\\
& \leq  c'  \sum_{j = 0}^{\infty}  \frac{1}{(1 + j)^{2 + \delta}} (1+j) n^{-3/2}  \leq c'' n^{-3/2}. 
\end{align*}
This proves \eqref{bound T}. 
\end{proof}

%We next recall the killed renewal estimate from \cite[Theorem 3.5]{GMX25}. 
From the killed renewal estimate in \cite[Theorem 3.5]{GMX25}, we have the following result.

\begin{proposition}\label{Prop-renewal-thm}
Assume \ref{proximality}-\ref{Condition-Furstenberg-Kesten}. 
Then, for any $b \geq 0$, there is a constant $c = c(b) >0$ such that for all $x \in \bb S^{d-1}_+$, all $t \in \bb R$ and $a >0$,
\begin{align}\label{eq:renewal.stopping.time}
\bb E \Big[ \sum_{n=0}^{\tau_{x, b} -1} \ds{1}_{[t,t+a]}(S_n^x)\Big]  \le  c \max\{a,1\}. 
\end{align}
Moreover, there is a constant $C >0$ such that for all $x \in \bb S^{d-1}_+$, all $t \in \bb R$ and $a >0$,
\begin{align}\label{eq:renewal.stopping.time002}
\bb E \Big[ \sum_{n=0}^{T -1} \ds{1}_{[t,t+a]}(S_n^x)\Big]  \le  C \max\{a,1\}.
\end{align}
\end{proposition}

\begin{proof}
Note that in \cite{GMX25} the result \eqref{eq:renewal.stopping.time} was formulated for a probability measure $\bb Q_x^\alpha$ defined as a measure change of $\bb P$ with respect to a parameter $\alpha>0$ and a harmonic function $r_\alpha$ on $\bb S^{d-1}_+$. 
In \cite{GMX25}, the change of measure is chosen so that the additive component $S_n$ has drift zero. The proof of Theorem 3.5 in \cite{GMX25} remains valid in this case by setting $\alpha=0$ and $r_\alpha$ being constant equal to 1; i.e., upon replacing $\bb Q_x^\alpha$ by $\bb P$.

Using \eqref{important inequality} and \eqref{eq:renewal.stopping.time} with $b = 2 \log \varkappa$, 
we get \eqref{eq:renewal.stopping.time002}. 
\end{proof}

From Proposition \ref{Prop-renewal-thm}, 
we deduce the exponential integrability of the overshoots that will be needed when applying optional stopping in the next section. 
%the following integrability result, which will be used in the next section. 

\begin{lemma}\label{Lem-integrability-S-tau}
Assume \ref{proximality}-\ref{Condition-Furstenberg-Kesten}. 
Let $\delta >0$ be as in \ref{Condition-moments-A-B}. 
Then there is a constant $c >0$ such that for all $x \in \bb S^{d-1}_+$, $\gamma \in (0, \delta/4)$ and $n \geq 1$,
\begin{align}\label{integrability-S-tau-x}
\bb E e^{- \gamma S_{\tau_x}^x} \leq c, 
\quad  \bb E e^{- \gamma S_{T}^x} \leq c 
\quad \mbox{and} \quad \bb E e^{- \gamma S_{n \wedge T}^x} \leq c. 
\end{align}
In addition, there is a constant $c >0$ such that for all $n \geq 1$ and $x \in \bb S^{d-1}_+$, $\bb P$-almost surely, 
\begin{align}\label{integrability-X-tau-x}
\min_{1 \leq i \leq d}  (X_{n}^x)_i  \geq c, 
\quad  
\min_{1 \leq i \leq d} (X_{\tau_x}^x)_i  \geq c 
\quad \mbox{and}  \quad  
\min_{1 \leq i \leq d}  (X_{T}^x)_i  \geq c. 
\end{align}
\end{lemma}

\begin{proof}
Since $\tau_x$ is the first time when $(S_n^x)_{n \geq 1}$ becomes negative, we have 
\begin{align*}%\label{}
\bb E e^{- \gamma S_{\tau_x}^x} 
\leq \bb E e^{ \gamma |\log |A_{\tau_x}^\top X_{\tau_x - 1}^x||} 
\leq  \bb E e^{ \gamma \log N(A_{\tau_x}^\top)} 
= \sum_{k = 1}^{\infty} \bb E  \left(  e^{ \gamma \log N(A_k^\top)}  \ds 1_{\{ \tau_x = k \}} \right). 
\end{align*}
Using H\"older's inequality and Lemma \ref{Lem-exit-time-estimate}, 
and taking $\gamma \in (0, \delta/4)$ with $\delta >0$ from condition \ref{Condition-moments-A-B},  we obtain 
\begin{align*}
 \sum_{k = 1}^{\infty} \bb E  \left(  e^{ \gamma \log N(A_k^\top)}  \ds 1_{\{ \tau_x = k \}}  \right)  
& \leq   \sum_{k = 1}^{\infty}   \bb E^{1/4}  \left(  e^{ 4 \gamma \log N(A_k^\top)}  \right)  \bb P^{3/4} (\tau_x = k)  \notag\\
& \leq  c  \sum_{k = 1}^{\infty}  \bb P^{3/4} (\tau_x = k)  \notag\\
& \leq c'  \sum_{k = 1}^{\infty}  k^{- 9/8} < \infty. 
\end{align*}
This proves the first inequality in \eqref{integrability-S-tau-x}.
The proof of the second and third ones is similar by replacing $\tau_x$ with $T$ and $n \wedge T$, respectively.

By Lemma \ref{lem equiv Kesten}, we see that $\bb P$-almost surely, $\min_{1 \leq i \leq d} (X_{n}^x)_i  \geq c$ for any $n \geq 1$.
Hence this also holds for almost surely finite stopping times $\tau_x$ and $T$. 
\end{proof}

\section{Proof of the main results}\label{Section-proof}

We now turn to the proofs of the main results (Theorems \ref{thm:main2} and \ref{thm:main3}). 
The argument proceeds in two steps. 
We first convert the tail problem for the invariant Radon measure into a Poisson equation 
along the centered Markov random walk and stop this equation at a descending ladder epoch. 
This yields uniform control of the slowly varying function appearing in the tail estimate of \cite{BPP21}. 
We then study vague accumulation points of suitably rescaled invariant measures 
and use the renewal theorem in Section \ref{Subsection renewal} for products of random matrices
 to identify their one-dimensional projections 
as multiples of the Haar measure $ds/s$. 

\subsection{A stopped Poisson identity}
%The proof strategy consists in transforming the study of tails of the invariant measure into a Poisson equation,
%and then applying the renewal theorem in Section \ref{Subsection renewal} for products of random matrices. 
%Let $f_{\Phi}$ and $\Psi_{\Phi}$ be from \eqref{def-function-f-Phi} and \eqref{Poisson-Psi-Phi}, respectively.
 The following result says that $f_{\Phi}$ and $\Psi_{\Phi}$ defined in \eqref{def-function-f-Phi} and \eqref{Poisson-Psi-Phi} 
 satisfy the Poisson equation for the stopped process. 

\begin{lemma}\label{Lem-identity-f-Phi}
%Suppose that $f_{\Phi}$ and $\Psi_{\Phi}$ satisfy the Poisson equation \eqref{Poisson-equation-f-Psi-001}.
%Then $f_{\Phi}$ and $\Psi_{\Phi}$ also satisfy the following Poisson equation: 
For any $s \in \bb R$ and $x \in (\bb S_+^{d-1})^\circ$, we have 
\begin{align}\label{eq:stopped Poisson}
\bb E f_{\Phi} \Big( e^{S_{T}^x - s} X_{T}^x \Big) 
- f_{\Phi} (e^{-s} x) 
= \bb E \sum_{k = 0}^{T -1} \Psi_{\Phi} \Big( e^{S_k^x - s} X_k^x \Big), 
\end{align}
for the stopping time $T$  defined by \eqref{def stopping T}. 
\end{lemma}

\begin{proof}
Since $(M_n(e^{-s} x), \scr F_n)$ is a martingale with mean $f_\Phi (e^{-s} x)$ (see Lemma \ref{Lem-martingale}), 
by the optional stopping theorem, 
we get, for any $n \geq 1$, 
\begin{align*}%\label{}
\bb E M_{n \wedge T} (e^{-s} x) = f_{\Phi} (e^{-s} x). 
\end{align*}
Hence, from \eqref{def martingale Mn} we get 
\begin{align*}%\label{}
\bb E f_{\Phi} \left( e^{S_{n \wedge T}^x - s} X_{n \wedge T}^x \right) 
- f_{\Phi} (e^{-s} x) = \bb E \sum_{k = 0}^{(n \wedge T) -1} \Psi_{\Phi} \left( e^{S_k^x - s} X_k^x \right). 
\end{align*}
By Lemma \ref{Lem-f-Phi-decay-rate-infinity}, we have that, for any $\gamma \in (0, \delta/2)$,
with $\delta$ from condition \ref{Condition-moments-A-B},
 and any function $\Phi$ in $C^1_c(\bb R_+^*)$, 
there is a constant $c >0$ (depending on $\Phi$) such that for any $n \geq 1$, $s \in \bb R$ and $x \in (\bb S_+^{d-1})^\circ$, 
\begin{align*}%\label{}
\left| f_\Phi \left(e^{S_{n \wedge T}^x-s} X_{n \wedge T}^x \right) \right|
\leq  c e^{\gamma s}  \Big(\min_{1 \leq i \leq d} (X_{n \wedge T}^x)_i  \Big)^{-\gamma}  e^{- \gamma S_{n \wedge T}^x}. 
\end{align*}
By Lemma \ref{lem equiv Kesten} and the fact that $T<\infty$, $\bb P$-almost surely,
there exists a constant $c>0$ such that, for all $n \geq 1$, we have $\min_{1 \leq i \leq d} (X_{n \wedge T}^x)_i \geq c$, $\bb P$-almost surely. 
Choose $p >1$ such that $p\gamma<\delta/4$. 
Hence, by \eqref{integrability-S-tau-x} of Lemma \ref{Lem-integrability-S-tau}, it follows that 
there exists a constant $c>0$ such that for any $n \geq 1$ and $x \in (\bb S_+^{d-1})^\circ$, 
\begin{align*}%\label{}
\sup_{n \geq 1} 
\bb E \left| f_\Phi\left(e^{S_{n\wedge T}^x-s}X_{n\wedge T}^x\right) \right|^p 
\leq  c e^{p\gamma s}  \,  \bb E \left( e^{- p \gamma S_{n \wedge T}^x} \right) \leq c < \infty, 
\end{align*}
so that the family $(f_\Phi (e^{S_{n\wedge T}^x-s}X_{n\wedge T}^x ))_{n \geq 1}$ of random variables is uniformly integrable. 
%\begin{align*}%\label{}
%\bb E \left| f_\Phi \left(e^{S_{n \wedge T}^x-s} X_{n \wedge T}^x \right) \right|
% \leq c e^{\gamma s} \bb E \left( \Big(\min_{1 \leq i \leq d} (X_{n \wedge T}^x)_i  \Big)^{-\gamma}  e^{- \gamma S_{n \wedge T}^x} \right). 
%%\leq c e^{\gamma s} \big[ \bb E \iota(A_1^\top)^{\gamma} \big]^n < \infty, 
%\end{align*}
%By the Cauchy-Schwarz inequality and Lemma \ref{Lem-integrability-S-tau}, we get that for all $\gamma \in (0, \delta/8)$ and $n \geq 1$, 
%\begin{align*}%\label{}
%\bb E \left(  \Big(\min_{1 \leq i \leq d} (X_{n \wedge T}^x)_i  \Big)^{-\gamma}  e^{- \gamma S_{n \wedge T}^x} \right)
% \leq  \bb E^{1/2} \left(  \Big(\min_{1 \leq i \leq d} (X_{n \wedge T}^x)_i  \Big)^{- 2\gamma}  \right) 
%  \bb E^{1/2} \left( e^{- 2\gamma S_{n \wedge T}^x} \right) 
%%& \leq  c \bb E^{1/2} \left( e^{- 2\gamma S_{T}^x} \right)
%\leq  c'. 
%% < \infty,  
%\end{align*}
%%where we have used the fact that $S_{n \wedge \tau_x}^x \geq 0$ on the set $\{n< \tau_x\}$. 
Therefore, we can apply the dominated convergence theorem to get that
\begin{align*}%\label{}
\lim_{n \to \infty} \bb E f_{\Phi} \left( e^{S_{n \wedge T}^x - s} X_{n \wedge T}^x \right) 
= \bb E f_{\Phi} \left( e^{S_{T}^x - s} X_{T}^x \right). 
\end{align*}
Notice that 
\begin{align*}%\label{}
\left| \bb E \sum_{k = 0}^{(n \wedge T) -1} \Psi_{\Phi} \left( e^{S_k^x - s} X_k^x \right) \right|
\leq  \bb E \sum_{k = 0}^{T -1} \left| \Psi_{\Phi} \left( e^{S_k^x - s} X_k^x \right) \right|. 
\end{align*}
Since the function $y \mapsto \Psi_{\Phi}(y)$
is directly Riemann integrable on $\Rdpint$ (see Lemma \ref{Lem-integrability}),
using Fubini's theorem and Proposition \ref{Prop-renewal-thm}, we obtain 
\begin{align*}%\label{}
\bb E \sum_{k = 0}^{T -1} \left| \Psi_{\Phi} \left( e^{S_k^x - s} X_k^x \right) \right|
& \leq \bb E \sum_{k = 0}^{T -1}  \sum_{m \in \bb Z} \sup_{r \in [e^m, e^{m+1}) } \sup_{x \in (\bb S_+^{d-1})^\circ}  |\Psi_{\Phi}( r x )|
\ds{1}_{[e^m, e^{m+1})}(e^{S_k^x - s})  \notag\\
& = \bb E \sum_{k = 0}^{T -1}  \sum_{m \in \bb Z} \sup_{r \in [e^m, e^{m+1}) } \sup_{x \in (\bb S_+^{d-1})^\circ}  |\Psi_{\Phi}( r x )|
\ds{1}_{[s+m, s+m+1)}(S_k^x) \notag\\
& =  \sum_{m \in \bb Z} \sup_{r \in [e^m, e^{m+1}) } \sup_{x \in (\bb S_+^{d-1})^\circ}  |\Psi_{\Phi}( r x )|
 \bb E \sum_{k = 0}^{T -1}  \ds{1}_{[s+m, s+m+1)}(S_k^x) \notag\\
& \leq  c  \sum_{m \in \bb Z} \sup_{r \in [e^m, e^{m+1}) } \sup_{x \in (\bb S_+^{d-1})^\circ}  |\Psi_{\Phi}( r x )| < \infty. 
\end{align*}
Therefore, again by the dominated convergence theorem, we get 
\begin{align*}%\label{}
\lim_{n \to \infty} 
\bb E \sum_{k = 0}^{(n \wedge T) -1} \Psi_{\Phi} \left( e^{S_k^x - s} X_k^x \right)
= \bb E \sum_{k = 0}^{T -1} \Psi_{\Phi} \left( e^{S_k^x - s} X_k^x \right). 
\end{align*}
The conclusion follows. 
\end{proof}

\subsection{Comparison with the slowly varying function}
We next compare $f_\Phi$ evaluated at the ladder direction $X_T^x$
with the slowly varying function $L$. 
This comparison transfers the radial estimate of \cite{BPP21}  
to the scalar projections that arise naturally from the Poisson equation.
 
We consider particular functions $\Phi \in C_c^1(\bb R_+^*)$ that approximate indicator functions. 
%We now choose a family of smooth cutoff functions $\Phi$. 
%Let 
%\begin{align*}%\label{}
%\kappa(t) = (1 - |t|) \ds{1}_{[-1, 1]}(t),  \quad t \in \bb R. 
%\end{align*}
%Recall that by \eqref{integrability-X-tau-x}, we have $\min_{1 \leq i \leq d} (X_{T}^x)_i  \geq c_0,$ $\bb P$-almost surely. 
%Its Fourier transform is given by $\frac{(\sin t)^2}{t^2}$.
For any $0 < a < b$ and $\ee \in (0, 1)$ such that $a - \ee >0$, %and $b > a c_0^{-1}$, 
%denote $\kappa_{\ee}(t) = \frac{1}{\ee} \kappa(\frac{t}{\ee})$ and $\Phi(t) =  \ds{1}_{[a,b]} * \kappa_{\ee}(t)$. 
choose $\Phi \in C^1_c(\bb R_+^*)$ such that 
\begin{align}\label{def function Phi}
 \ds{1}_{[a, b]}(t) \leq  \Phi(t) \leq   \ds{1}_{[a-\ee, b + \ee]}(t). 
\end{align}

\begin{lemma}\label{Lem compare strip}
Let $Y$ be a $\bb{S}^{d-1}_+$-valued random variable with the property that there is $c_0>0$ such that $\min_{1 \leq i \leq d} Y_i \geq c_0$ a.s.; and let $\Phi$ be as in \eqref{def function Phi} such that $a - \ee >0$ and $b > a c_0^{-1}$. 
Then there are constants $c_1, c_2 > 0$ (which depend on $a,b$) such that, for any $s \in \bb R_+$ and $x \in \bb S^{d-1}_+$, we have, $\bb P$-almost surely, 
\begin{align}\label{inequality compare strip.general}
c_1 \leq \frac{f_{\Phi} (e^{- s} Y)}{ L(e^{s})}  \leq c_2. 
\end{align}
In particular, it holds that for any $s \in \bb R_+$ and $x \in \bb S^{d-1}_+$, $\bb P$-almost surely, 
\begin{align}\label{inequality compare strip}
	c_1 \leq \frac{f_{\Phi} (e^{- s} X_{T}^x)}{ L(e^{s})}  \leq c_2. 
\end{align}
\end{lemma}

\begin{proof}
By \eqref{def-function-f-Phi}, we have
\begin{align*}%\label{}
f_{\Phi} (e^{- s} x) = \int_{\bb R^d_+}  \Phi\big( \langle e^{- s}x, z \rangle \big) \nu(dz) 
\leq  \nu \left( \left\{  z \in \bb R^d_+:  \langle x, z \rangle \in  e^s [a - \ee,  b + \ee]  \right\} \right)
\end{align*}
and 
\begin{align*}%\label{}
f_{\Phi} (e^{- s} x) 
\geq  \nu \left( \left\{  z \in \bb R^d_+:  \langle x, z \rangle \in e^s [a, b]  \right\} \right). 
\end{align*}
Recall that by \eqref{integrability-X-tau-x}, we have $\min_{1 \leq i \leq d} (X_{T}^x)_i  \geq c_0,$ $\bb P$-almost surely. Hence, $Y=X_T$ is a legitimate choice.
Employing this bound and 
using \eqref{Important-Inequality}, we get that, $\bb P$-almost surely, 
\begin{align*}%\label{}
|z| \geq \langle Y, z \rangle \geq  \langle \bf 1, z \rangle  \min_{1 \leq i \leq d} (Y)_i 
=  |z| \min_{1 \leq i \leq d} (Y)_i \geq c_0 |z|. 
\end{align*}
Denote by $R_{a, b} = \{ x \in \bb R^d_+: |x| \in [a,b] \}$ the annulus between radii $a$ and $b$.
The result \eqref{eq:result.BPP} can be reformulated as follows: 
there exist a positive slowly varying function $L$ on $[1, \infty)$ and constants $a_0, b_0, c>0$ such that for any $t \geq 1$, 
%The mentioned result can be reformulated, denoting by $R_{a, b} = \{ x \in \bb R^d_+: |x| \in [a,b] \}$ the annulus between radii $a$ and $b$, 
\begin{align}\label{eq:result.BPP002}
L(t) \le \nu \left(  t \, R_{a_0, b_0}  \right) \le c L(t).
\end{align}
Consequently, we obtain 
\begin{align*}%\label{}
f_{\Phi} (e^{- s} Y)
& \leq \nu \left( \left\{  z \in \bb R^d_+:   e^s (a-\ee) \leq  \langle Y, z \rangle \leq e^s (b+\ee)  \right\} \right)  \notag\\
& \leq  \nu \left( \left\{  z \in \bb R^d_+:   e^s (a-\ee) \leq |z| \leq c_0^{-1} e^s (b+\ee)  \right\} \right)  \notag\\
& =  \nu \left( e^s R_{a-\ee, \, c_0^{-1} (b+\ee)} \right)
\end{align*}
and 
\begin{align*}%\label{}
f_{\Phi} (e^{- s} Y)
& \geq \nu \left( \left\{  z \in \bb R^d_+:   e^s a \leq  \langle Y, z \rangle \leq e^s b  \right\} \right)  \notag\\
& \geq  \nu \left( \left\{  z \in \bb R^d_+:   c_0^{-1} e^s a \leq |z| \leq  e^s b  \right\} \right) \notag\\
& =  \nu \left( e^s R_{c_0^{-1} a, \, b} \right). 
\end{align*}
Now we can find finitely many annuli of the form $s_r R_{a_0, b_0}$, $1 \leq r \leq k$, 
such that $$R_{a-\ee, \, c_0^{-1} (b+\ee)} \subset \bigcup_{r = 1}^k s_r R_{a_0, b_0}$$
for some nonnegative real numbers $s_1, \ldots, s_k$ (depending on $a_0$ and $b_0$). 
Then, using \eqref{eq:result.BPP002},  we obtain 
\begin{align*}%\label{}
\nu \left( e^s R_{a-\ee, \, c_0^{-1} (b+\ee)} \right) \leq \sum_{r = 1}^k  \nu \left( s_r e^s R_{a_0, \, b_0} \right) \leq c \sum_{r = 1}^k L(s_r e^s). 
\end{align*}
Since $L$ is slowly varying at infinity, this proves the upper bound in \eqref{inequality compare strip.general}. 
Similarly, to show the lower bound, we can also find finitely many annuli of the form $s_r R_{c_0^{-1} a, \, b}$, $1 \leq r \leq k$, 
such that $R_{a_0, \, b_0} \subset \cup_{r = 1}^k s_r R_{c_0^{-1} a, \, b}$
for some positive real numbers $s_1, \ldots, s_k$. 
Therefore, 
\begin{align*}%\label{}
\nu \left( e^s R_{a_0, \, b_0} \right) \leq \sum_{r = 1}^k  \nu \left( s_r e^s R_{c_0^{-1} a, \, b} \right), 
% \leq c \sum_{r = 1}^k  L(s_r e^s), 
\end{align*}
so that 
\begin{align*}%\label{}
\max_{1 \leq r \leq k} \nu \left( s_r e^s R_{c_0^{-1} a, \, b} \right) \geq  \frac{1}{k} \nu \left( e^s R_{a_0, \, b_0} \right). 
\end{align*}
This implies that 
\begin{align*}%\label{}
\nu \left( e^s R_{c_0^{-1} a, \, b} \right) \geq  \frac{1}{k}  \min_{1 \leq r \leq k} \nu \left( s_r^{-1} e^s R_{a_0, \, b_0} \right) \geq c_k \min_{1 \leq r \leq k} L(s_r^{-1} e^s). 
\end{align*}
This concludes the proof of \eqref{inequality compare strip.general}. 
\end{proof}

\begin{proposition}\label{Prop L bounded}
%Assume \ref{proximality}-\ref{Condition-Irreducibility}. Then 
The slowly varying function $L$ in \eqref{eq:result.BPP} is bounded. % and hence constant at infinity. 
\end{proposition}

\begin{proof}%[Proof of Theorem \ref{thm:main2}]
Denote $T_{1} = T$ (see \eqref{def stopping T}) and define recursively the ladder times:
\begin{align*}%\label{}
T_{n} = \inf \{ k > T_{n-1}: \log \|A_k^\top \cdots A_{T_{n - 1} + 1}^\top \|  < 0 \}. 
\end{align*}
%By \cite[Lemma 2, equation (3.14)]{Kes74}, 
\textbf{Step 1}. Observe that $(X_{T_{n}}^x)_{n \geq 1}$ is a Markov chain with respect to the filtration $\scr F_{T_n}$ on the compact space $\bb S_+^{d-1}$, 
with starting point $x$ and transition operator given as follows: for any bounded measurable function $f$ on $\bb S_+^{d-1}$, 
\begin{align*}%\label{}
Q f(x) = \bb E f \Big( (A_T^\top \cdots A_1^\top) \cdot x \Big) 
= \sum_{k = 1}^\infty \bb E\left(  f \Big( (A_k^\top \cdots A_1^\top ) \cdot x \Big) \ds 1_{\{T = k\}} \right). 
\end{align*} 
From the second representation of $Q$ together with Lemma \ref{Lem-exit-time-estimate}, we see that $Q$ satisfies the weak Feller property: 
it maps from continuous functions to continuous functions on $\bb S_+^{d-1}$. 
Therefore, by Theorem 12.0.1 (i) in \cite{MT93}, %{\color{magenta}by Meyn and Tweedie, 1993, Theorem 12.0.1 (i), }
there exists an invariant probability measure $\rho$ for the Markov chain $(X_{T_{n}}^x)_{n \geq 1}$
satisfying that, for any Borel measurable set $A \subset \bb S_+^{d-1}$, 
\begin{align*}%\label{stationarity rho 001}
\int_{\bb S_+^{d-1}}  \bb P (X_{T}^x \in A) \rho(dx) = \rho(A). 
\end{align*}
In particular,
\begin{align}\label{identity measure rho}
\int_{\bb S_+^{d-1}}    f_{\Phi} (e^{-s} x)   \rho(dx)
=   \int_{\bb S_+^{d-1}}  \bb E   f_{\Phi} (e^{-s} X_{T}^x)  \rho(dx). 
\end{align}
Note that it follows from  \eqref{integrability-X-tau-x} that $\mathrm{supp}(\rho) \subset (\bb S^{d-1}_+)^\circ$. 
Integrating \eqref{eq:stopped Poisson} of Lemma \ref{Lem-identity-f-Phi} with respect to $\rho$
and using the identity \eqref{identity measure rho}, we have that, for any $s \in \bb R$,
%\begin{align*}%\label{}
%\bb E f_{\Phi} \left( e^{S_{T}^x - s} X_{T}^x \right) 
%- f_{\Phi} (e^{-s} x) = \bb E \sum_{k = 0}^{T -1} \Psi_{\Phi} \left( e^{S_k^x - s} X_k^x \right). 
%\end{align*}
%Integrating both sides with respect to the invariant measure $\rho$, we obtain
\begin{align}\label{pf-main-equation-001}
& \int_{\bb S_+^{d-1}}  \bb E    f_{\Phi} \left( e^{S_{T}^x - s} X_{T}^x \right) \rho(dx)
  -  \int_{\bb S_+^{d-1}}  \bb E   f_{\Phi} (e^{-s} X_{T}^x)  \rho(dx)  \notag\\
& = \int_{\bb S_+^{d-1}} \bb E    \sum_{k = 0}^{T -1} \Psi_{\Phi} \left( e^{S_k^x -s} X_k^x \right) \rho(dx). 
\end{align}
Integrating \eqref{pf-main-equation-001} with respect to $s$ over $[-N, t]$ and using Fubini's theorem, we get
\begin{align}\label{Pf-equation-001}
&  \int_{\bb S_+^{d-1}}   \int_{-N}^t  \bb E    f_{\Phi} \left( e^{S_{T}^x - s} X_{T}^x \right)  ds \rho(dx)
  -  \int_{\bb S_+^{d-1}}  \int_{-N}^t  \bb E   f_{\Phi} (e^{-s} X_{T}^x)  ds \rho(dx) \notag\\
& = \int_{\bb S_+^{d-1}} \int_{-N}^t  \bb E    \sum_{k = 0}^{T -1} \Psi_{\Phi} \left( e^{S_k^x -s} X_k^x \right)  ds \rho(dx). 
\end{align}
Considering the left-hand side of \eqref{Pf-equation-001}, we apply first a change of variable and Fubini's theorem to obtain
\begin{align*}%\label{}
\int_{\bb S_+^{d-1}}  \int_{-N}^t    \bb E   f_{\Phi}  \left( e^{S_{T}^x - s} X_{T}^x \right)  ds  \rho(dx) 
 =  \bb E  \int_{\bb S_+^{d-1}}  \int_{ -N - S_{T}^x }^{t - S_{T}^x }    f_{\Phi}  \left( e^{- s} X_{T}^x \right)  ds \rho(dx). 
\end{align*}
Hence,
\begin{align*}%\label{}
\text{L.H.S. of \eqref{Pf-equation-001}} 
% \int_{\bb S_+^{d-1}} \int_{-N}^t    \bb E    f_{\Phi} \left( e^{S_{T}^x - s} X_{T}^x \right)  ds  \rho(dx)
%  - \int_{\bb S_+^{d-1}}   \int_{-N}^t  f_{\Phi} (e^{-s} x) ds  \rho(dx)   \notag\\
& =  \bb E  \int_{\bb S_+^{d-1}}  \int_{ -N - S_{T}^x }^{t - S_{T}^x }    f_{\Phi}  \left( e^{- s} X_{T}^x \right)  ds  \rho(dx)
  -  \bb E  \int_{\bb S_+^{d-1}}  \int_{-N}^t  f_{\Phi} (e^{-s} X_{T}^x)  ds \rho(dx)  \notag\\
& = \bb E  \int_{\bb S_+^{d-1}}  \int_{t}^{t - S_{T}^x}   f_{\Phi}  \left( e^{- s} X_{T}^x \right)  ds \rho(dx)
  -  \bb E  \int_{\bb S_+^{d-1}}  \int_{-N}^{-N - S_{T}^x}    f_{\Phi}  \left( e^{- s} X_{T}^x \right)  ds \rho(dx). 
\end{align*}
Substituting this into \eqref{Pf-equation-001} and rearranging terms, we arrive at
\begin{align}\label{Pf-equation-change-of-variable}
& \bb E  \int_{\bb S_+^{d-1}}  \int_{t}^{t - S_{T}^x}    f_{\Phi}  \left( e^{- s} X_{T}^x \right)  ds \rho(dx) \notag\\
& =  \bb E  \int_{\bb S_+^{d-1}}  \int_{-N}^{-N - S_{T}^x}    f_{\Phi}  \left( e^{- s} X_{T}^x \right)  ds  \rho(dx)
 +   \int_{\bb S_+^{d-1}}  \int_{-N}^t  \bb E    \sum_{k = 0}^{T -1} \Psi_{\Phi} \left( e^{S_k^x -s} X_k^x \right)  ds \rho(dx). 
\end{align}

\textbf{Step 2}. In this step, we are going to show that the right-hand side of \eqref{Pf-equation-change-of-variable} remains bounded when taking limits $N\to \infty$ and then $t \to \infty$.
We start with the first term on the right-hand side of \eqref{Pf-equation-change-of-variable}. 
By Lemma \ref{Lem-f-Phi-decay-rate-infinity}, 
we have 
%$|f_\Phi(y)| \le \frac{c}{ (\min_{1 \leq i \leq d} \overline y_i)^{\gamma} } |y|^{-\gamma}$, so that, 
for any $s \in \bb R$ and $x\in (\bb S^{d-1}_+)^\circ$
$
|f_{\Phi} (e^{- s} x)| \leq   (\min_{1 \leq i \leq d}  x_i)^{-\gamma}c e^{\gamma s}. % \leq c' e^{\gamma s},
$
%where in the last inequality we used $x \in B \subset \mathrm{int}(\bb S^{d-1}_+)$. 
Hence, by Lemma \ref{Lem-integrability-S-tau}, we see that 
\begin{align*}%\label{}
& \bb E \left| \int_{\bb S_+^{d-1}}  \int_{-N}^{-N - S_{T}^x}    f_{\Phi}  \left( e^{- s} X_{T}^x \right) ds \rho(dx)   \right|  \notag\\
& \leq c  \bb E \left| \int_{\bb S_+^{d-1}}  \int_{-N}^{-N - S_{T}^x}   
  \max_{1 \leq i \leq d}  \left( \langle e_i,  X_{T}^x \rangle^{-\gamma} \right)  e^{\gamma s}  ds  \rho(dx)  \right|  \notag\\
& = \frac{c}{\gamma}  e^{- \gamma N}  \bb E 
\int_{\bb S_+^{d-1}}  \max_{1 \leq i \leq d}  
\left( \langle e_i,  X_{T}^x \rangle^{-\gamma} \right) \left( e^{ - \gamma S_{T}^x} - 1 \right) \rho(dx)  \notag\\ 
& \leq c_1 e^{- \gamma N}. 
\end{align*}
It follows that 
\begin{align}\label{Pf-bound-N-S-tau-001}
\lim_{N \to \infty}  \bb E \left|  \int_{\bb S_+^{d-1}}  \int_{-N}^{-N - S_{T}^x}    f_{\Phi}  \left( e^{- s} X_{T}^x \right)  ds  \rho(dx) \right|  = 0. 
\end{align}
For the second term on the right-hand side of \eqref{Pf-equation-change-of-variable}, 
letting $N \to \infty$ and applying the renewal theorem (Proposition \ref{Prop-renewal-thm}), 
we get
\begin{align}\label{Pf-bound-N-S-tau-002}
\limsup_{t \to \infty}   \int_{\bb S_+^{d-1}}  \int_{-\infty}^t  
\bb E    \sum_{k = 0}^{T -1} \Psi_{\Phi} \left( e^{S_k^x -s} X_k^x \right)  ds \rho(dx)
 \leq c_2 < \infty, 
\end{align}
using that, by Lemma \ref{Lem-integrability002}, 
the function $(x, t) \mapsto \int_{-\infty}^t \Psi_{\Phi} (e^{-s} x) ds$ is directly Riemann integrable 
on $(\bb S_+^{d-1})^\circ \times \bb R$. 

\textbf{Step 3}. 
For brevity, denote 
\begin{align*}%\label{}
I(t) = \bb E  \int_{\bb S_+^{d-1}}  \int_{t}^{t - S_{T}^x}    f_{\Phi}  \left( e^{- s} X_{T}^x \right)  ds \rho(dx). 
\end{align*}
Taking into account  \eqref{Pf-equation-change-of-variable}, 
\eqref{Pf-bound-N-S-tau-001} and \eqref{Pf-bound-N-S-tau-002}, we have obtained that 
\begin{align}\label{bounded funct I(t)}
\limsup_{t \to \infty}  I(t) 
%\limsup_{t \to \infty}  \bb E  \int_{\bb S_+^{d-1}}  \int_{t}^{t - S_{T}^x}    f_{\Phi}  \left( e^{- s} X_{T}^x \right)  ds \rho(dx) 
 \leq c_2 < \infty. 
\end{align}
By a change of variable, we have 
\begin{align*}%\label{}
\int_{t}^{t - S_{T}^x}  f_{\Phi} (e^{- s} X_{T}^x) ds 
& = \int_{0}^{ - S_{T}^x }  f_{\Phi} \left( e^{- (t+s)} X_{T}^x \right) ds  \notag\\
& = L(e^{t}) \int_{0}^{ - S_{T}^x }  \frac{f_{\Phi} (e^{- (t+s)} X_{T}^x)}{ L(e^{t+s}) }  \frac{ L(e^{t+s}) }{ L(e^{t}) } ds.
\end{align*}
Combining the last two displayed formulas and using Lemma \ref{Lem compare strip}, 
we obtain that there exists a constant $c_3 \in (0, \infty)$ such that 
\begin{align*}%\label{}
\frac{I(t)}{ L(e^{t}) }
& =  \bb E  \int_{\bb S_+^{d-1}}   \int_{0}^{ - S_{T}^x } 
  \frac{f_{\Phi} (e^{- (t+s)} X_{T}^x)}{ L(e^{t+s}) }  \frac{ L(e^{t+s}) }{ L(e^{t}) } ds   \rho(dx)  \notag\\
& \geq  c_3  \bb E  \int_{\bb S_+^{d-1}}   \int_{0}^{ - S_{T}^x }   \frac{ L(e^{t+s}) }{ L(e^{t}) } ds   \rho(dx). 
\end{align*}
By Fatou's lemma and the fact that $L$ is a slowly varying function at infinity, it follows that 
\begin{align}\label{pf apply Fatou L(t)}
\liminf_{t \to \infty}  \frac{I(t)}{ L(e^{t}) } 
\geq  c_3   \int_{\bb S_+^{d-1}}  \bb E(- S_{T}^x)  \rho(dx). 
\end{align}
Since $S_T^x = \log |A_T^\top  \cdots A_1^\top x| \leq \log \| A_T^\top  \cdots A_1^\top \| < 0$ for any $x \in \bb S_+^{d-1}$, 
we get $\bb E(- S_{T}^x) > 0$, so that 
%$- S_{\tau_x}^x >0$ and $|Ax| \leq \|A\|$, we get $- S_{T}^x > 0$.  
 $\int_{\bb S_+^{d-1}} \bb E(- S_{T}^x) \rho(dx) > 0$. 
Therefore, from \eqref{bounded funct I(t)} and \eqref{pf apply Fatou L(t)}, we can conclude that 
%from the previous inequality that, for $\rho$-almost surely $x \in \bb S_+^{d-1}$, 
\begin{align*}%\label{}
\limsup_{t \to \infty} L(e^{t}) \leq c < \infty, 
\end{align*}
as desired. 
%By \eqref{eq:result.BPP002}, we have $\liminf_{t \to \infty} L(e^{t}) \geq c' \limsup_{t \to \infty} L(e^{t})$ for some $c' >0$. 
%This forces the slowly varying function $L$ to be bounded. 
\end{proof}

\begin{corollary}\label{Cor compare strip}
Let $Y$ be a $\bb{S}^{d-1}_+$-valued random variable with the property that there is $c_0>0$ such that $\min_{1 \leq i \leq d} Y_i \geq c_0$ a.s.; and let $\Phi$ be as in \eqref{def function Phi} such that $a - \ee >0$ and $b > a c_0^{-1}$. Then there are constants $c_1, c_2 > 0$ (which depend on $a,b$) such that, for any $s \in \bb R$ and $x \in \bb S^{d-1}_+$, we have, $\bb P$-almost surely, 
\begin{align*}
c_1 \leq {f_{\Phi} (e^{- s} Y)}  \leq c_2. 
\end{align*}
\end{corollary}

\begin{proof}
This is a direct consequence of Lemma \ref{Lem compare strip}, using that $L$ is bounded due to Proposition \ref{Prop L bounded},
and that $L$ is positive on $\bb R_+$ due to \cite[Theorem 1.1]{BPP21}. 
\end{proof}

\subsection{Proof of Theorems \ref{thm:main2} and \ref{thm:main3}}
We now pass from quantitative bounds to the asymptotic properties of the invariant measure. 
The preceding estimates imply vague relative compactness of the family of dilated measures. 
We shall show that every nontrivial accumulation point 
has multiplicatively invariant one-dimensional projections along directions in $\supp \pi$.

%Given a measure $\eta$ on $\bb R^d_+ \setminus\{0\}$, we write $\overline{\eta}$ 
%for its push-forward measure on $(0,\infty)$ under the map $z \mapsto |z|$. 
%Then, the following change of variable formula holds: for any continuous functions $f$ with compact support on $(0,\infty)$, 
%\begin{align}\label{def overline nu s}
% \int_0^{\infty} f(s) \overline{\eta}(ds) = \int_{\bb R^d_+ \setminus\{0\}} f( |z|) \eta(dz).
%\end{align}
%With this notation \eqref{eq:result.BPP002} becomes: for any $s > 0$, 
%\begin{align*}
%L(e^{s}) \le \nu_s \left( R_{a, b}  \right) = \overline \nu_s([a, b]) \le c L(e^{s}). 
%\end{align*}
%We write $\lambda(ds) = \frac{ds}{s}$ for the Haar measure on the multiplicative group $(\bb R_+^*, \cdot)$.

\begin{proof}[Proof of Theorem \ref{thm:main2}]
Given a compact set $K \subset \bb R_+^d \setminus\{0\}$, we can cover it by a finite union of sets of the form 
$$ K \subset \bigcup_{j=1}^m  r_j  R_{a, b},$$
where $(r_j)_{1\le j \le m}$ are suitable constants, 
$R_{a, b} = \{ x \in \bb R^d_+: |x| \in [a,b] \}$ and
 $[a, b]$ is the set from \eqref{eq:result.BPP}. 
%	Recall from \eqref{def nu s} and \eqref{def overline nu s} the definitions of $\nu_s$ and $\overline \nu_s$, respectively. 
Define a family $(\nu_s)_{s >0}$ of dilations of $\nu$ by
\begin{align*}%\label{def nu s}
	\int_{\bb R^d_+ \setminus\{0\}} f(y) \nu_s(dy) = \int_{\bb R^d_+ \setminus\{0\}} f(e^{-s}z) \nu(dz).
\end{align*}
	By Proposition \ref{Prop L bounded}, it follows that there is $C<\infty$ such that for all $s>0$, 
	$$\nu_s(K)=\nu(e^s K) \le C.$$
	Hence, the family $(\nu_s)_{s >0}$ is vaguely relatively compact. Fix an arbitrary sequence $s_k \to \infty$, then there is a subsequence $s_n \to \infty$ of $(s_k)$, such that $\nu_{s_n}$ converges vaguely to a Radon measure $\eta$ on $\bb R_+^d \setminus\{0\}$.

\textbf{Step 1}. We are going to prove invariance properties of $\eta$. 
Consider an arbitrary nonnegative function $\Phi \in C^1_c(\bb R_+^*)$ and define for $x \in \Rdpint$, 
 $$F_\Phi(x) = \int_{\bb R^d_+} \Phi \big( \langle x, z \rangle \big) \eta(dz).$$
 Recalling the definition of 
 $$ f_\Phi(x)= \int_{\bb R^d_+} \Phi\big( \langle x, z \rangle \big) \nu(dz),$$
 we have 
 $$ F_\Phi(x)= \lim_{n \to \infty} f_\Phi(e^{-s_n}x),$$
 using the definition of vague convergence and that for each $x$, 
 the mapping $z \mapsto \Phi(\scal{x}{z})$ is continuous and has compact support  
 whenever all entries of $x$ are strictly positive (since then $\scal{x}{z}$ and $|z|$ are comparable).
	
By Fatou's lemma and the Poisson equation \eqref{Poisson-Psi-Phi}, we have
\begin{align*}
\bb E \Big[ F_\Phi(e^{S_1^x}X_1^x) \Big]  
& = 	\bb E \Big[ \liminf_{n \to \infty} f_\Phi(e^{S_1^x-s_n}X_1^x) \Big] \\
& \leq  \liminf_{n \to \infty} \bb E \Big[ f_\Phi(e^{S_1^x-s_n}X_1^x) \Big] \\
& = \liminf_{n \to \infty} \Big( \Psi_\Phi(e^{-s_n}x) + f_\Phi(e^{-s_n}x) \Big) \\
& = F_\Phi(x),
\end{align*}
where we have used  Lemma \ref{Lem-integrability} showing that $\lim_{|x| \to 0} \Psi_\Phi(x)=0$.
This proves that 
$$H_n(x)= F_\Phi \big( e^{S_n^x}X_n^x \big)$$
is a nonnegative supermartingale. 
Hence there is a nonnegative random variable $H_\infty(x)$ 
with $\bb E [H_\infty(x)]\le F_\Phi(x)$ and we have the $\bb P$-almost sure convergence
\begin{equation}\label{eq:conv.h.H}
\lim_{n \to \infty} H_n(x)=\lim_{n \to \infty} F_\Phi(e^{S_n^x}X_n^x)=H_\infty(x). 
\end{equation}
Recall that $\pi$ is the unique stationary probability measure for the Markov chain $X_n$ on $\bb S^{d-1}_+$
 and that $\supp(\pi)\subset (\bb S^{d-1}_+)^\circ$ by Lemma \ref{lem equiv Kesten}. 
It is proved in \cite[Corollary 3.3]{KM15} that 
there is $x \in (\bb S^{d-1}_+)^\circ$ such that for any open set $E \subset (\bb S^{d-1}_+)^\circ$ with $\pi(E)>0$ 
and every nonempty open set $G \subset \bb R$, 
$$ \bb P \big( (X_n^x, S_n^x) \in E \times G \text{ infinitely often} \big)=1.$$
Together with the $\bb P$-almost sure convergence in \eqref{eq:conv.h.H} and the continuity of $F_\Phi$, 
this implies that there is a constant $c_0$ such that $F_\Phi(ru)=c_0$ for all $u \in \mathrm{supp}(\pi)$ and  $r>0$.

Hence, we obtain that for any compactly supported $\Phi$, $x \in \supp \pi$ and all $r>0$,
$$\int_{\bb R^d_+} \Phi(\langle x, z \rangle) \eta(dz) = \int_{\bb R^d_+} \Phi(r \langle x, z \rangle) \eta(dz)$$
showing that the pushforward measure $\overline{\eta}_x$ of $\eta$ under $z \mapsto \langle x, z \rangle$ is invariant under multiplication, 
hence is equal to a multiple of the Haar measure $\lambda(ds)= \frac{ds}{s}$ on the multiplicative group $(\bb R_+^*, \cdot)$. That is,
$$\int_0^\infty \Phi(s) \overline{\eta}_x(ds)
 = \int_{\bb R_+^d \setminus\{0\}} \Phi(\langle x, z \rangle) \eta(dz) = C' \int_0^\infty \Phi(s) \lambda(ds) = C' \int_0^\infty \Phi(s) \frac{ds}{s}.$$
Hence, we have proved that, 
for any sequence $s_k \to \infty$, 
there is a subsequence $(s_n)_{n \geq 1}$ and $C'>0$ (depending on this subsequence) such that for any $x \in \supp \pi$
and any nonnegative function $\Phi \in C^1_c(\bb R_+^*)$,  
\begin{align}\label{eq:limit subsequence}
\lim_{n \to \infty} f_{\Phi}(e^{-s_n} x) = C' \int_0^\infty \Phi(s) \frac{ds}{s}. 
\end{align}

\textbf{Step 2}. 
We now want to bound the possible range of $C'$. 
Therefore, we evaluate \eqref{eq:limit subsequence} for a function $\Phi$ given by \eqref{def function Phi}. 
Then, by Corollary \ref{Cor compare strip}, for any $x \in \supp \pi$ (which is contained in the interior of $\bb S^{d-1}_+$), it holds that there are $0<c_1 \leq  c_2<\infty$ such that  $c_1\leq f_\Phi(e^{-s}x)\leq c_2$. Writing $c_3=\int_0^\infty \Phi(s) \lambda(ds)$,  we have that $c_1/c_3 \leq C' \leq c_2/c_3$. 

Therefore, for any sequence $s_k \to \infty$, there is a subsequence $(s_n)_{n \geq 1}$ such that for any $x \in \supp \pi$, 
and any nonnegative function $\Phi \in C^1_c(\bb R_+^*)$,  
\begin{align*}%\label{}
\frac{c_1}{c_3} \int_0^\infty \Phi(s) \frac{ds}{s} 
\leq \liminf_{n \to \infty} f_{\Phi}(e^{-s_n} x) \leq \limsup_{n \to \infty} f_{\Phi}(e^{-s_n} x) \leq \frac{c_2}{c_3} \int_0^\infty \Phi(s) \frac{ds}{s}. 
\end{align*}
Consequently, the same bounds hold for any sequence $s_n \to \infty$. 
This proves Theorem \ref{thm:main2}.
\end{proof}

\begin{proof}[Proof of Theorem \ref{thm:main3}]
Arguing as in the proof of Lemma \ref{Lem compare strip}, we obtain that for any  direction $x \in (\bb S^{d-1}_+)^\circ$ and any $0<a<b$, 
%satisfying the constraint $b/a > \varkappa d$ 
\begin{align*}%\label{}
\nu \left( \left\{  z \in \bb R^d_+:   t a \leq |z| \leq  t b  \right\} \right) 
 \geq  \nu \left( \left\{  z \in \bb R^d_+:   t a \leq  \langle x, z \rangle \leq c_0 t b  \right\} \right)
 \end{align*}
and 
\begin{align*}%\label{}
\nu \left( \left\{  z \in \bb R^d_+:    t a \leq |z| \leq  t b  \right\} \right)
\leq  \nu \left( \left\{  z \in \bb R^d_+:   c_0 t a \leq  \langle x, z \rangle \leq t b  \right\} \right), 
\end{align*}
where $c_0 = \min_{1 \leq i \leq d} x_i$. 
By Lemma \ref{lem equiv Kesten}, we have that $\min_{1 \leq i \leq d} x_i \geq \frac{1}{\varkappa d}$ for any $x \in \supp \pi$. 
% there are constants $c_1, c_2$
%	such that for sufficiently large $t>0$,
%	\begin{align}\label{compare norms 001}
%	c_1 \leq \frac{ \nu \left(  \left\{ z  \in \bb R^d_+ \, : \, |z| \in t [a,b]  \right\}  \right) }{ \nu \left(  \left\{ z  \in \bb R^d_+ \, : \, \scal{x}{z} \in t [a,b]  \right\}  \right) }  \leq c_2. 
%	\end{align}
%	%a sufficient condition for $c_1$ to be positive is $b/a > \varkappa d$. 
%	%Note that the condition $b> c_0^{-1} a$ can be satisfied by considering $t' [sa, sb]$ instead of $t [a, b]$, by simply writing $t = t' s$. 
Then the result follows from Theorem \ref{thm:main2}. % and \eqref{compare norms 001}. 
\end{proof}

%%%%%%%%%%%%%%%%%%%%%%%%%%%%%%%%%%%%%%%%%%%%%%%%%%%
%%%%%%%%%%%%%%%%%%%%%%%%%%%%%%%%%%%%%%%%%%%%%%%%%%%

\end{document}